\documentclass[]{elsarticle}

\usepackage{lineno,hyperref}

\usepackage{graphicx}
\usepackage{bm}
\usepackage{multirow}
\usepackage{amsmath}
\usepackage{amsfonts}
\usepackage{amssymb}
\usepackage{amsthm}
\usepackage{secdot}
\usepackage{subcaption}
\usepackage{lscape}
\usepackage{xcolor}
\usepackage{enumitem}
\usepackage{array}

\journal{Computer-Aided Design Journal}

\theoremstyle{plain}
\newtheorem{theorem}{Theorem}[section]
\newtheorem{lemma}[theorem]{Lemma}

\theoremstyle{definition}
\newtheorem{definition}[theorem]{Definition}

\theoremstyle{remark}
\newtheorem{remark}[theorem]{Remark}

\theoremstyle{assumption}
\newtheorem{assumption}[theorem]{Assumption}

\numberwithin{equation}{section}
\numberwithin{theorem}{section}

\newcommand{\llbracket}{\left[\!\left[}
\newcommand{\rrbracket}{\right] \! \right]}

\newcommand{\polyorder}{k}

\newcommand{\polyfunction}{\phi}

\newcommand{\vertex}{p}

\begin{document}

\begin{frontmatter}

\title{Implementation of Worsey-Farin Splines with Applications to Solution Transfer}

\author{Logan Larose\corref{mycorrespondingauthor}}

\cortext[mycorrespondingauthor]{Corresponding author}
\ead{lfl5340@psu.edu}

\author{David M. Williams} 

\address{Department of Mechanical Engineering, The Pennsylvania State University, University Park, Pennsylvania 16802}

\address{Laboratories for Computational Physics and Fluid Dynamics, Naval Research Laboratory, Washington, DC 20375}

\fntext[fn1]{Distribution Statement A: Approved for public release. Distribution is unlimited.}

\begin{abstract}
This work primarily focuses on providing full implementation details for Worsey-Farin (WF) spline interpolation over tetrahedral elements. While this spline space is not new and the theory has been covered in other works, there is a lack of explicit and comprehensive implementation details, which we hope to provide. In this paper, we also demonstrate the effectiveness of the WF-spline space through a simple target application: solution transfer. Moreover, we derive an error estimate for the WF spline-based, solution transfer process. We conduct numerical experiments quantifying the conservative nature and order of accuracy of the transfer process, and we present a qualitative evaluation of the visualization properties of the smoothed solution. Additionally, in our study of conservation, we demonstrate how adaptive numerical quadrature rules on the tetrahedron used in conjunction with global $L_2$-projection can improve the conservation of the solution transfer process. 
\end{abstract}


\end{frontmatter}


\section{Introduction}
\label{sec;introduction}

To date, splines have been rigorously studied, with contributions to the field since at least Schoenberg's coining of the term in 1946~\cite{schoenberg1988contributions}, and perhaps, even earlier developments before the modern terminology began circulating according to Farin~\cite{farin2001curves} and de Boor~\cite{de1986approximation}. The last 60 years of splines have been largely driven by applications, with some of the earliest motivations of spline research arising in the field of mechanical design. Major contributors in this field are de Casteljau for Citro\"en~\cite{de1959outillage}\cite{de1963courbes}, B\'ezier for Renault, de Boor for General Motors~\cite{de1968least}, Sabin for British Aircraft Corporation~\cite{sabin1971existing}, and Ferguson for Boeing~\cite{ferguson1986construction}. Many of these names are mentioned multiple times throughout this paper because of their prolific contributions to the field. The increase in the usage of splines parallels the increased usage of computers as a design tool because of their ability to smoothly join at control points of airfoils and vehicle bodies. Thus, splines are unaffected by the Runge effect that is associated with their predecessor, interpolating-polynomials.

The desirable interpolation properties of splines produced another area of application: the approximation of partial differential equations (PDEs). Much of the aforementioned design work is based on B-splines, a term introduced by Schoenberg for their use as basis functions~\cite{schoenberg1969cardinal}. Bazilevs et al.~\cite{bazilevs2006isogeometric}, Hughes et al.~\cite{hughes2005isogeometric}, and Evans et al.~\cite{evans2009n} are major contributors to the development of IsoGeometric Analysis (IGA), which relies on B-splines. The premise behind IGA was to use the same basis-functions for both CAD representation and PDE approximation to reduce the error introduced by modeling with separate bases for design and analysis. Splines were already a proven design tool prior to this application, however this added credence to their use in approximation. Later developments in IGA introduce variations/specializations of splines to respond to IGA's limitations on handling
complex geometries and local refinement, such as T-splines~\cite{scott2012local}, THB-splines~\cite{giannelli2012thb}, LR-Splines~\cite{dokken2013polynomial}, and S-splines~\cite{li2019s}. More generally, splines have made their way into tensor product approximation spaces such as those in discontinuous Galerkin (DG) methods~\cite{bressan2019approximation}. For a detailed discussion of the use of splines to approximate PDEs on triangulations, see~\cite{speleers2005powell},~\cite{pelosi2017splines} and the references therein. 

Another field of study that opened up in the late 1980s and early 1990s was the smooth visualization of piecewise polynomial functions and derivatives on a mesh. This effort was led by  Farin~\cite{farin1986triangular}, Patrikalakis~\cite{patrikalakis1989representation}, Loop~\cite{loop1990generalized}, 
and Peters~\cite{peters1993smooth}, to name a few. However, this is by no means a comprehensive list and the interested reader should check~\cite{bajaj2017splines} and references therein for a more rigorous review of these developments. 

We end our brief review of the applications of splines by introducing their use for solution transfer. Solution transfer in this context refers to transferring or interpolating the solution on one mesh to another dissimilar mesh. Solution transfer is a well-developed area of scientific computing, see~\cite{farrell2009conservative,farrell2011conservative} and references therein, with linear interpolation being perhaps the most well-known approach. However, 
the poor conservation~\cite{alauzet20073d,alauzet2010p1} of the solution provided by linear transfer has led to conservative interpolation schemes that meet the rising demand for accurate and conservative transfer. For example, Farrell et al.~\cite{farrell2009conservative,farrell2011conservative}, and Alauzet and Mehrenberger~\cite{alauzet2010p1,alauzet2016parallel} offer schemes that outperform linear interpolation on the basis of conservation. In our previous work~\cite{larose2025spline}, we explain these methods in greater detail, as well as introduce spline-based solution transfer. We used Hsieh-Clough-Tocher (HCT) splines to transfer solutions between 2D meshes of triangular elements. Our  transfer strategy achieved acceptable conservation on fine-grids, while also producing an intermediate solution that was smooth and optimal for visualization. We now wish to demonstrate the solution transfer capabilities of Worsey-Farin (WF) splines between 3D meshes of tetrahedral elements, and present the most comprehensive description of the WF interpolation procedure to date. 

\subsection{Literature Review}

Of the notable contributors to the field of splines, we narrow our focus to those whose work is most relevant to the present paper. We begin by considering the work of Farin, who built off of the work of B\'ezier~\cite{bezier1970emploi}\cite{bezier1977essai}, De Casteljau~\cite{de1959courbes}, and Sabin~\cite{sabin1977use} to produce B\'ezier polynomials over triangles and generalized theory on $\mathcal{C}^r$ piecewise polynomials~\cite{farin1980bezier}. Farin showed that the general relation of polynomial degree $d$ to desired smoothness $r$ is $d = 4r+1$ on triangulations,  (see~\cite{alfeld1987minimally}~\cite{lai2007spline} for additional details). This relation allows for locally constructed elements which are stable; however as $r$ increases, the associated high-degree polynomials are not favorable from an implementation standpoint. This is why subdividing (splitting) a triangular element---often referred to as a triangular macroelement---into smaller triangular elements is advantageous, as it reduces the order of the polynomial necessary to achieve the same degree of smoothness. Farin's work is one of multiple camps of development occurring in the late 1970s and early 1980s which established interpolating splines over triangles. Elsewhere, in France, Bernadou and Hassan built upon the work of Argyis et al.~\cite{argyris1968tuba} and the split provided by Hsieh, Clough, and Tocher~\cite{clough1965finite} to produce basis functions for a $\mathcal{C}^1$ class of finite elements~\cite{bernadou1980triangles}; in England, Powell and Sabin developed a different (alternative) triangular split which enabled $\mathcal{C}^1$ approximation~\cite{powell1977piecewise}. Powell and Sabin's method uses quadratic polynomials for interpolation, but requires splitting a triangle into 6 subtriangles, meanwhile Bernadou and Hassan's implementation of HCT elements requires splitting a triangle into 3 subtriangles and using cubic polynomials to form the interpolant.  

Many of the developments for subdividing triangular elements were rapidly extended to 3D tetrahedral elements. This was motivated by the 3D nature of many practical applications, as well as the relationship between polynomial order and desired smoothness, which becomes $d = 8r+1$ on tetrahedra. With this in mind, Alfeld extended the work of Hsieh, Clough, and Tocher to 3D, via splitting a macrotetrahedron into 4 subtetrahedra by connecting the vertices to the centroid~\cite{alfeld1984trivariate}, and then placing a quintic polynomial on each subtetrahedron. Worsey and Farin offered an alternative  split~\cite{worsey1987n} to Alfeld's which allows for a $\mathcal{C}^1$ interpolant using cubic polynomials by increasing the number of subtetrahedra from Alfeld's 4 to 12---which we will explain in great detail in our implementation section. Shortly afterwards, Worsey and Piper extended Powell-Sabin's work on triangles to tetrahedra producing the Worsey-Piper split~\cite{worsey1988trivariate}  which uses quadratic polynomials to approximate trivariate functions over 24 subtetrahedra. This split requires several restrictive geometric conditions in order to maintain $\mathcal{C}^{1}$-continuity. In fact, Worsey and Piper~\cite{worsey1988trivariate}, and Sorokina and Worsey~\cite{sorokina2008multivariate} showed that a sufficient condition for $\mathcal{C}^{1}$-continuity is an \emph{acute} mesh. Here, an acute mesh is one in which each tetrahedra contains its circumsphere center, and each triangular face also contains its circumcircle center. Acute meshes are very difficult to generate in practice, unless the mesh vertices are the points of a uniform rectilinear grid. Alas, it is clearly up to the developer which 3D  spline to choose based on polynomial degree, number of subelements, and complexity of geometric constraints. We believe that the Worsey-Farin split offers a very effective compromise between these factors.  

There are a few researchers who have undertaken significant efforts to document and make the accomplishments of their peers more accessible. One of the earliest examples is de Boor's \emph{A Practical Guide to Splines}~\cite{de1978practical} and later \emph{B(asic)-Spline Basics}~\cite{de1986b}, both of which thoroughly discuss splines in the context of computer-aided design. An updated and comprehensive view on the field of splines is offered by Lai and Schumaker in \emph{Spline Functions on Triangulations}~\cite{lai2007spline}, and Schumaker's \emph{Spline Functions: Basic Theory}~\cite{schumaker2007spline}. Additionally, Schumaker's \emph{Spline Functions: Computational Methods}~\cite{schumaker2015spline} and  \emph{Spline Functions: More Computational Methods}\cite{schumaker2024spline} are  practical resources for those looking to implement splines. 

\subsection{Motivation}

Previously, in~\cite{larose2025spline} we detailed our proposed solution transfer process between 2D meshes using HCT splines. Broadly speaking, this process consists of three steps: (1) synchronization (i.e.~smoothing) of the numerical solution on the original mesh, (2) spline interpolation of the smoothed solution, and (3) projection of the spline solution to the new mesh. We refer the interested reader to~\cite{larose2025spline} for implementation details regarding steps (1) and (3), as these steps are very similar in both the 2D and 3D settings. In fact, for our extension to 3D, this process and its motivations remain the same, however the interpolating HCT splines in step (2) are replaced by WF splines. In this work we will demonstrate the transfer of a discontinuous solution on an unstructured tetrahedral mesh of $k=1$ or $k=2$ elements to a similar-sized, but non-conforming, unstructured tetrahedral mesh of elements of the same order. We will also present a comprehensive description of WF splines, which we will use in our transfer process to produce a smooth and continuous surrogate solution of our discontinuous solution. To the best of our knowledge, this is the only full description of WF splines available currently. As mentioned previously, the original theory of this spline-space is presented by Worsey and Farin (for whom the spline space is named) in~\cite{worsey1987n}; however clear implementation details are lacking in their work. Another description of this space is presented briefly by Lai and Schumaker in Section 18.4 of~\cite{lai2007spline}, and perhaps the most condensed and updated information useful for tetrahedral spline implementation is given in Chapter 6 of~\cite{schumaker2024spline}. Our description is informed by sections 2.1-2.6, 6.2, 15.3-15.7, and 18.4 of~\cite{lai2007spline}, as well as the consultation of T. Sorokina of Towson University who has collaborated with many of the formative developers mentioned throughout this section. 

\subsection{Overview of the Paper}

The format of this work is as follows: we begin by providing a detailed description of our implementation of WF-based solution transfer in Section~\ref{sec;Implementation}. Next, Section~\ref{sec;Theoretical Results} presents theoretical results which characterize the accuracy and robustness of the solution transfer method. Thereafter, Section~\ref{sec;Results} presents the results of numerical experiments which (a) assess the order of accuracy of the method, and (b) assess the impact of adaptive quadrature and global $L_2$-projection on the conservation of the method. We end the section by verifying the  visualization properties of our method. Section~\ref{sec;Conclusion} summarizes the findings of this work.

\section{Implementation Details}
\label{sec;Implementation}

\subsection{Preliminaries on Splines}

For those unfamiliar with splines, we will introduce some standard terminology and concepts that are crucial for understanding the implementation. 
\begin{definition}[Bernstein Basis Polynomials]
    \textit{The cubic Bernstein basis polynomials, $B_{ij\ell m}(\lambda)$, associated with a tetrahedron, $T$, take the form,}
    \begin{align}
        B_{ij\ell m}(\lambda) = \frac{3!}{i!j!\ell! m!} \lambda_1^i \lambda_2^j \lambda_3^{\ell} \lambda_4^m, \qquad i + j + \ell + m = 3, \qquad i,j,\ell,m \in [0,3],
    \end{align}
    \textit{and form a partition of unity, i.e.} 
    \begin{align}
        \sum\limits_{i+j+\ell+m=3}B_{ij\ell m}(\lambda) = 1, \qquad \forall \lambda\in\mathbb{R}^4. 
    \end{align}
    \textit{Note that $\lambda = (\lambda_1,\lambda_2,\lambda_3,\lambda_4)$ is a set of linear polynomials, trivially, or barycentric coordinates over $T$, such that $\lambda_1+\lambda_2 +\lambda_3 + \lambda_4 = 1$. 
    } \label{def:bbpoly}
\end{definition}
\noindent Definition~\ref{def:bbpoly} allows us to introduce the Bernstein-B\'ezier form (also called Bernstein representation or simply \emph{B-form}), which we use to form our WF-spline space over tetrahedron $T$.
\begin{definition}[Bernstein-B\'ezier Form]
    \textit{The WF space on each tetrahedron $\mathcal{WF}(T)$ can be expressed as 12 non-overlapping subtetrahedra, $T = T^{\alpha = 1} \cup \hspace{0.25mm} T^{\alpha = 2} \cup \hspace{0.25mm} T^{\alpha = 3} \cdots \cup \hspace{0.25mm} T^{\alpha = 11} \cup \hspace{0.25mm} T^{\alpha = 12}$ equipped with cubic polynomials ($\polyorder=3$), $\polyfunction^{\alpha}$, with coefficients, $c_{ij\ell m}^{\alpha}$, such that}
    \begin{align}
        \polyfunction^{\alpha} = \sum_{i+j+\ell+m=\polyorder} c_{ij\ell m}^{\alpha} B_{ij\ell m}(\lambda^{\alpha}), 
    \end{align}
    \textit{where each $B_{ij\ell m}(\lambda^{\alpha})$ is a Bernstein polynomial following definition~\ref{def:bbpoly}, and}
    \begin{align*}
         \lambda^{\alpha} = \left(\lambda_{1}^{\alpha}, \lambda_{2}^{\alpha}, \lambda_{3}^{\alpha}, \lambda_{4}^{\alpha}  \right),
    \end{align*}
    \textit{are barycentric coordinates defined on each subtetrahedron} $T^{\alpha}\in T$.
    \label{def:B-form}
\end{definition}
\noindent Definition~\ref{def:B-form} introduces \emph{coefficients}, or B-coefficients (also referred to as B\'ezier ordinates), $c_{ij\ell m}^{\alpha}$. B-coefficients are associated with the Bernstein basis polynomials and are given extensive attention in this work because they are crucial for evaluating the smoothed solution which arises during the solution transfer process. These coefficients are associated with \emph{domain points}, $\xi^{ij\ell m}$, which are the physical locations where the B-coefficients are applied---similar to degrees of freedom. For an arbitrary tetrahedron, the domain points of a cubic polynomial can be found using
\begin{equation}
    \xi^{{ij\ell m}} = \frac{i\vertex^1 + j\vertex^2+\ell \vertex^3+m \vertex^4}{3}, \qquad i+j+\ell+m = 3,
    \label{eqn:domain_points}
\end{equation}
where $\vertex^1, \ldots, \vertex^4$ are the tetrahedron vertices.
Figure~\ref{fig:domain_points} illustrates the domain points for a cubic polynomial on a tetrahedron. 
\begin{figure}[h!]
    \centering
    \includegraphics[width=0.8\linewidth]{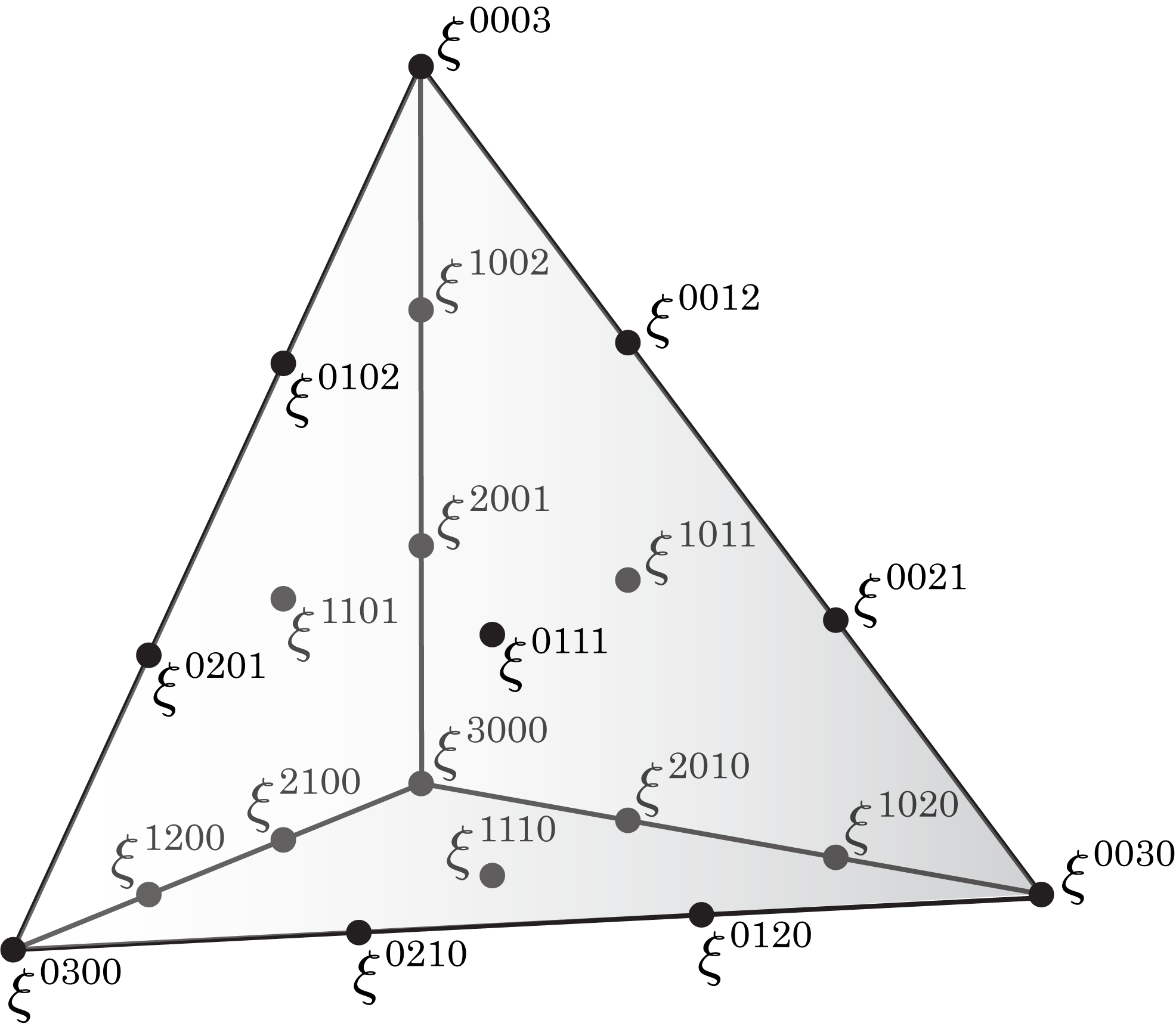}
    \caption{The domain points for a cubic polynomial on a tetrahedron generated using equation~\ref{eqn:domain_points}.}
    \label{fig:domain_points}
\end{figure}

\subsection{Worsey-Farin Interpolation and Solution Transfer}

We begin by introducing some notation. Consider an initial mesh over the domain $\Omega \in \mathbb{R}^3$, which we will refer to as the \textit{source mesh} henceforth. We denote the source mesh by $\mathcal{T}_a$. Now, consider a second mesh over the same domain, $\Omega$, which was produced by some adaptation process such that it is not conforming nor hierarchical with respect to the source mesh. We will refer to this second mesh as the \textit{target mesh} henceforth, and we denote it by $\mathcal{T}_b$. Our objective is to transfer the solution $u_{h_a}$ from $\mathcal{T}_a$ to $\mathcal{T}_b$.

As we have discussed previously, the process for transferring the solution can be described in three steps:

\begin{itemize}
    \item Averaging of the solution and its gradients on the source mesh. The averaged solution is denoted by $E(u_{h_a})$, where $E(\cdot)$ is the averaging or \emph{synchronization} operator. This process is carried out in a pointwise fashion.
    \item Spline interpolation of the averaged solution and its gradients on the source mesh. This process constructs a continuous version of $E(u_{h_a})$, using the pointwise values computed in the previous step. 
    \item $L_2$-projection of the spline interpolated solution from the source mesh to the target mesh. The $L_2$-projection operator is denoted by $I_{h_b}(\cdot)$.
\end{itemize}
These steps were originally discussed  in~\cite{larose2025spline}, and we provide a visual summary of the steps in Figure~\ref{fig:transfer}.
\begin{figure}
    \centering
    \includegraphics[width=1\linewidth]{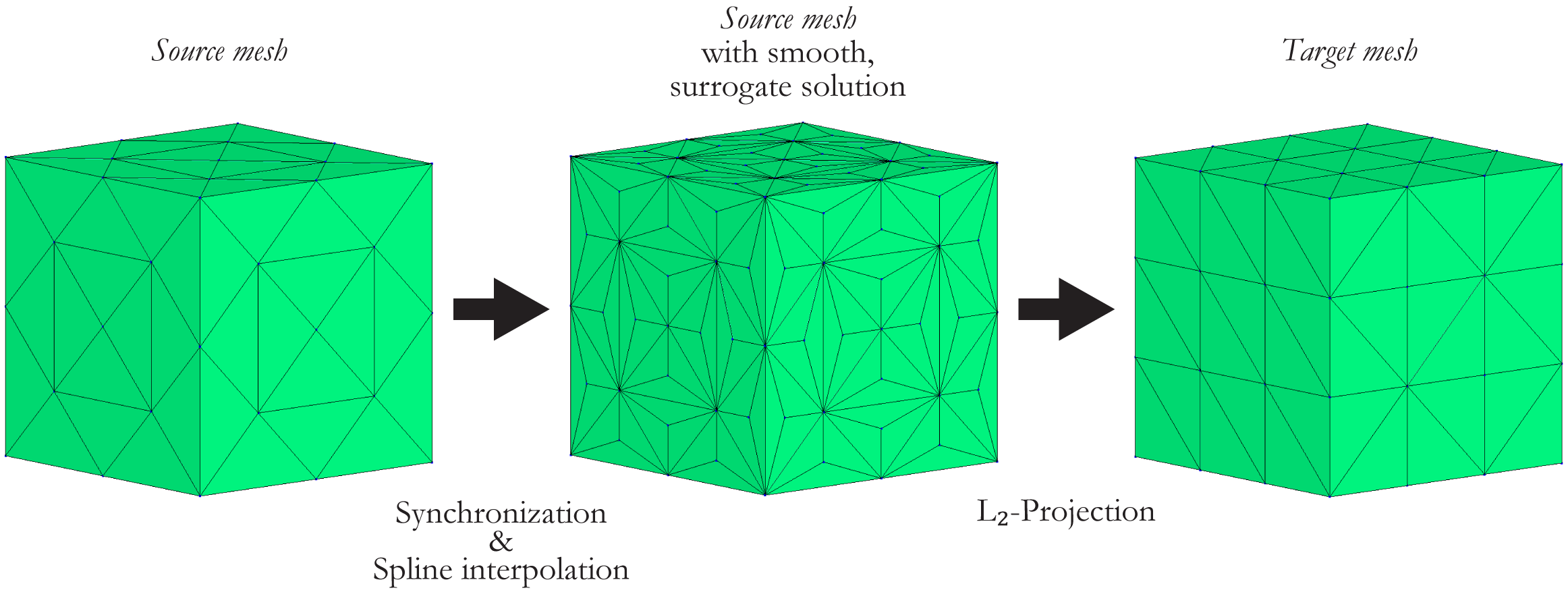}
    \caption{A visual representation of our proposed 3-step transfer process.}
    \label{fig:transfer}
\end{figure}

In accordance with the purpose of this paper, we will now discuss in more detail the spline interpolation process in 3D. During the transfer process, the spline must be constructed on each element of the source mesh $\mathcal{T}_{a}$. Thereafter, the spline must be evaluated (or sampled) in order to perform the required $L_2$-projection on to $\mathcal{T}_{b}$. It is important to note that, in practice, the $L_2$-projection is performed  using quadrature rules defined on the elements of $\mathcal{T}_b$. With these ideas in mind, the primary technical objective becomes very simple: compute the values of the splines on elements in $\mathcal{T}_{a}$ at a set of quadrature points which are taken from elements in $\mathcal{T}_{b}$. Once the values of the splines are computed at the quadrature points, then the $L_2$-projection can be immediately performed. In what follows, we explain how to evaluate the splines at the quadrature points.

Let us assume that each element of $\mathcal{T}_b$ contains a certain number of quadrature points. In order to compute the appropriate spline values at these points, we must identify the elements in $\mathcal{T}_{a}$ that contain the quadrature points. The locations of these points are found using a bounding volume hierarchy (BVH) method combined with a matrix determinant method for finding points that lie on or near element boundaries. The details of this are included in~\cite{larose2025spline}, and extend naturally to 3D.

Once a quadrature point is found within a source element of $\mathcal{T}_a$, that source element---which we call the macro-tetrahedron with respect to its subtetrahedra---must be split into 12-subtetrahedra. The split is determined by:
\begin{enumerate}
    \item Identifying the tetrahedra which share a face with the macro-tetrahedron.
    \item Calculating the incenter of the macro-tetrahedron as well as  its face-neighbors in accordance with~\cite{klein2020insphere}.
    \item Determining the point of intersection on each face of the macro-tetrahedron using the lines formed between the incenter of the tetrahedra and the incenters of its face-neighbors. This intersection becomes the site for a \emph{Clough-Tocher} split on each face of the macro-tetrahedron.
    \item[NOTE:]\textit{Steps 1-3 of the procedure are illustrated for a single face neighbor in Figure~\ref{fig:wf_split}.}
    \item 4-subtetrahedra are formed by connecting the vertices of the macro-tetrahedron to its incenter. This is an \emph{Alfeld}-type split.
    \item Each of the 4-subtetrahedra are subdivided into 3 parts by connecting their vertices to their respective facial split points (from step 3). 
    \item[NOTE:] \textit{Steps 4 and 5 of the procedure are illustrated in Figure~\ref{fig:macro_tet}.} 
\end{enumerate}

\begin{figure}[h!]
    \centering
    \includegraphics[width=1.0\linewidth]{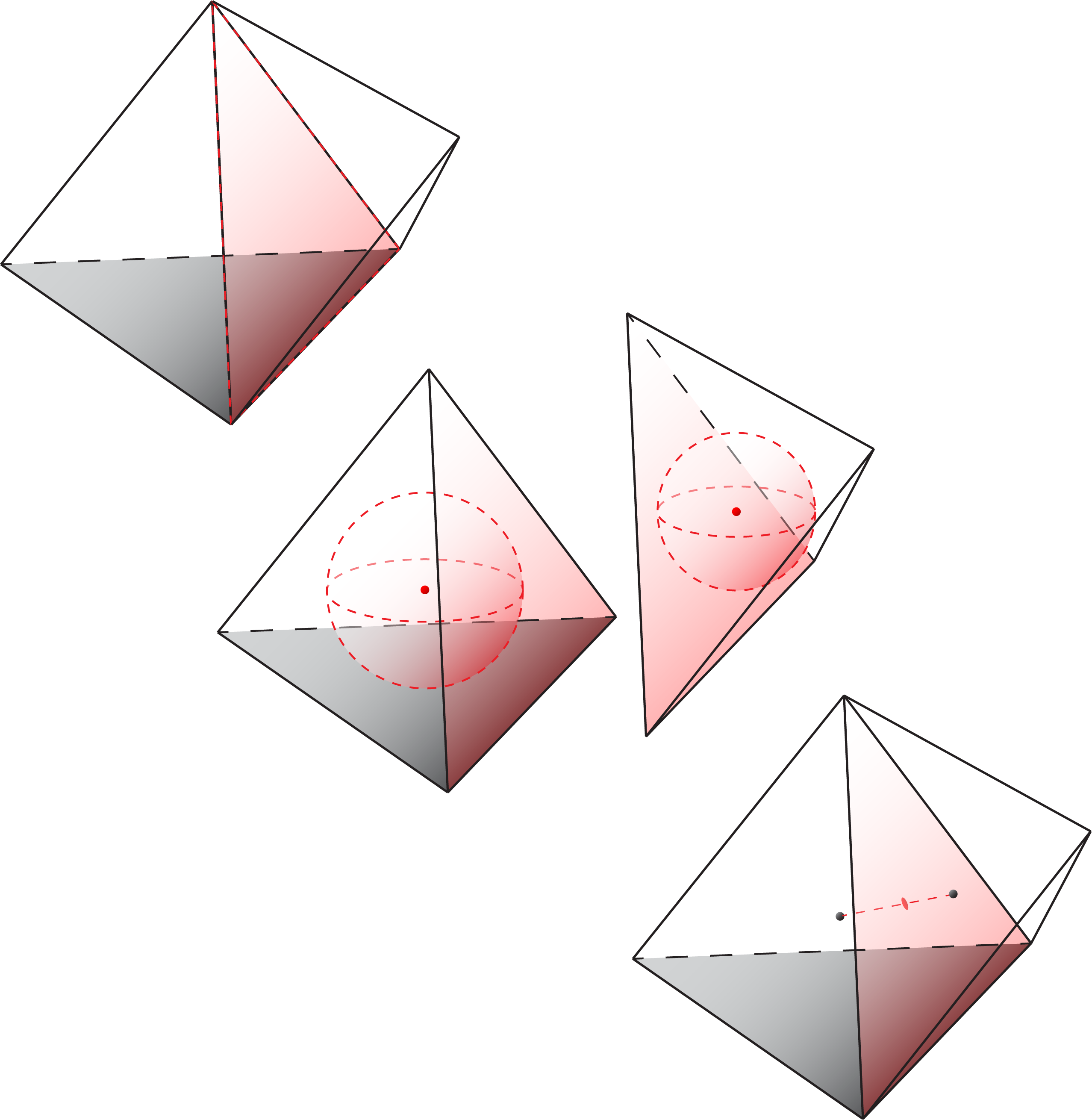}
    \caption{The Worsey-Farin split process shown for a single face neighbor. Top left shows the identification of the face neighbor; center shows an exploded view of the calculation of the incenter of each tetrahedron; bottom right shows the face split point (red ellipse) produced by intersection of the line between incenters and the shared face.}
    \label{fig:wf_split}
\end{figure}

\begin{figure}[h!]
    \centering
    \includegraphics[width=1.0\linewidth]{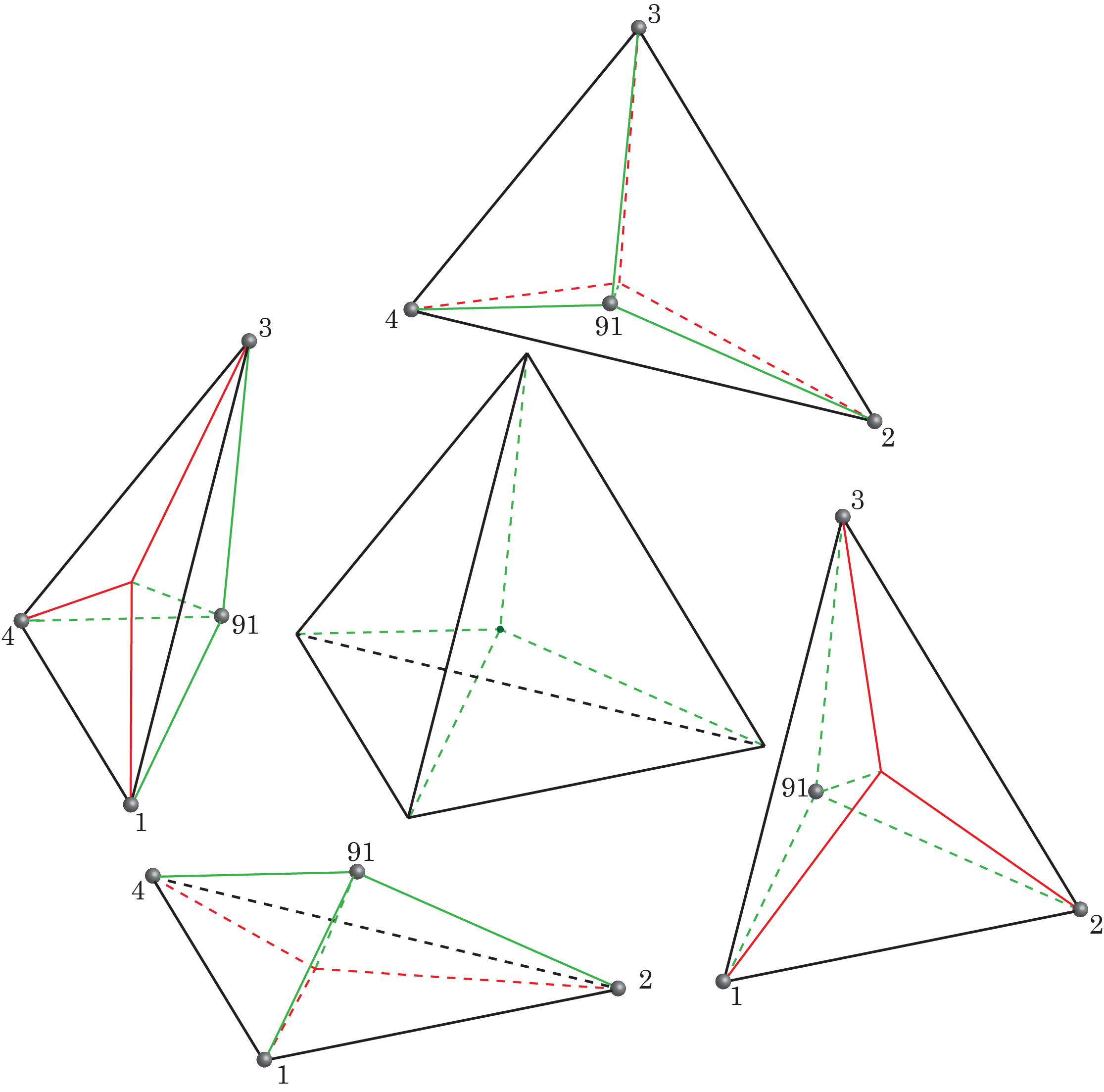}
    \caption{The macro-tetrahedron shown in the center with internal edges illustrated by dotted lines. Subtetrahedra are produced by connecting the vertices of the macro-tetrahedron to the incenter (shown by green lines). Subsequent subtetrahedra are produced by connecting the vertices of the subtetrahedra to the facial split points (shown by red lines). The incenter and facial split points were found from the process shown in Figure~\ref{fig:wf_split}. The numbering of nodes in this figure is simply the degree-of-freedom numbering for the WF spline. This numbering is explained in the subsequent text.}
    \label{fig:macro_tet}
\end{figure}
With the subtetrahedra defined, we can construct the interpolating splines by endowing each subtetrahedra with cubic domain points and B-coefficients as illustrated in Figure~\ref{fig:domain_points}. The advantage of using Bernstein-Bezier form is that the algebraic smoothness can be understood geometrically---especially in lower dimensions. A short proof of the continuity condition for smooth joins of polynomials on two tetrahedra is given in Section 15.14 of~\cite{lai2007spline}. This condition is crucial for the construction of the $\mathcal{C}^1$ WF-spline space over a macro-element, such that polynomials defined over each subelement join continuously along faces and edges. Part of the impact of meeting this condition means that coefficients are shared along interior faces and edges between subtetrahedra.\footnote{This grants $\mathcal{C}^0$ continuity. Additional restrictions on coefficients are discussed in Section 15.14 of~\cite{lai2007spline} and are necessary for $\mathcal{C}^1$ continuity.} Because of the shared coefficients, this necessitates using a mapping to select the appropriate coefficients from the 91 total coefficients on the macroelement to associate with the local domain points shown in Figure~\ref{fig:domain_points}, which are defined relative to the subtetrahedra. Table~\ref{tab:Mappings} contains the mappings for all subtetrahedra (see section~\ref{sec:mass_conservation_tests} of the Numerical Results).

With the subdivisions produced for the source element, the location of each quadrature point within one of the 12-subtetrahedra must be determined in order to select the appropriate mapping of coefficients for evaluation. This determination is made using a matrix determinant method and iterating first over the 4-subtetrahedra based on the Alfeld-split, and then their subsequent subtetrahedra. The final step is to evaluate the splines at the quadrature point using the coefficients defined over the subtetrahedra via the de Casteljau algorithm. The de Casteljau algorithm is defined in Section 15.6 of~\cite{lai2007spline}, but also has been rigorously described by the spline and computer science community since de Casteljau's development of this algorithm in 1959~\cite{de1959outillage}\cite{de1963courbes}\cite{boehm1999casteljau}. We restate it here for completeness and acknowledge that this definition borrows heavily from notes of Peter Alfeld~\cite{AlfeldWebPage}, but is extended to the trivariate case for our purposes.
\begin{definition}[The de Casteljau Algorithm]
\textit{The polynomial, $\polyfunction$, defined on the tetrahedron, $T$, can be evaluated at any point, $P$ contained in $T$, as follows:}
    \begin{align}
        \polyfunction(P) = c_{0000}^{\gamma} 
    \end{align}
    \textit{where},
    \begin{align}
        c_{ij\ell m}^{0} = c_{ij\ell m}, \qquad i + j + \ell + m = \gamma, 
    \end{align}
    \textit{and}
    \begin{align}
        c_{ij\ell m}^{r+1} = \lambda_1 c_{i+1,j,\ell, m}^{r}+\lambda_2 c_{i,j+1,\ell, m}^{r}+\lambda_3 c_{i,j,\ell+1, m}^{r}+\lambda_4 c_{i,j,\ell, m+1}^{r},& \\ \qquad i + j + \ell + m + r + 1 = \gamma, \quad r = 0, \cdots, \gamma - 1. \nonumber
    \end{align}
    \textit{In this context, the superscripts of the coefficients, $c$, indicate the step of the algorithm.} 
\end{definition}
For the sake of brevity, all 91 coefficients and diagrams for all 12 subtetrahedra are given in the Appendix. 
\section{Theoretical Results}
\label{sec;Theoretical Results}

In this section, we discuss the important mathematical properties of the WF-based  solution transfer process. First, we discuss our motivation, and then introduce some relevant notation and assumptions. Thereafter, we discuss the properties of the WF-based smoothing operator $E(\cdot)$, and the $L_2$-interpolation operator $I_{h}(\cdot)$. Finally, we combine the previous properties and assumptions in order to construct an error estimate and a discrete maximum principle for our solution transfer process. 

\subsection{Motivation}

Although the accuracy of cubic spline interpolation is well-understood, the accuracy of the entire solution-transfer process has \emph{not} been characterized. In other words, the combination of smoothing, interpolation, and projection steps need not preserve the accuracy of the solution. Hence, the order of accuracy will be rigorously established in what follows.

Throughout the remainder of this section, we will closely follow the error estimate in Section 3 of~\cite{larose2025spline}. The proof in the present section differs from the original proof, as the original  proof was limited to solution transfer with 2D, HCT splines. In contrast, the present proof requires the precise handling of the WF spline degrees of freedom in 3D. Furthermore, it constructs mathematical bounds for the jumps in the solution across triangular faces, as opposed to jumps across 2D edges.

\subsection{Theoretical Preliminaries}

In this section, we will assume that the numerical solution is produced by a finite element method. This enables us to leverage the rich, mathematical formalism associated with such methods. With this in mind, we begin in classical fashion by introducing the piecewise-polynomial space
\begin{align*}
    V^{\polyorder}_{\mathcal{T}} &= \left\{u \in L_{2}(\Omega) : u_T = u|_T \in \mathcal{P}_{\polyorder} (T) \quad \forall T \in \mathcal{T} \right\}.
\end{align*}
Here, $\mathcal{P}_{\polyorder}(T)$ is the space of polynomials of degree $\leq \polyorder$ on the tetrahedral element $T$. In a natural fashion, this element belongs to the triangulation $\mathcal{T}$. 

Furthermore, we can construct a $\mathcal{C}^{1}(\Omega)$, piecewise-polynomial spline space
\begin{align*}
    W_{\mathcal{T}} &= \left\{w \in H^{2}(\Omega) : w_T = w|_T \in \mathcal{WF} (T) \quad \forall T \in \mathcal{T} \right\}.
\end{align*}
Here, $\mathcal{WF}$ is the space of piecewise-cubic Worsey-Farin splines. 

\begin{assumption}[Truncation Error] Consider a partial differential equation with an exact solution $u \in H^{q}(\Omega)$, where $q$ is a positive integer. The following error estimate holds for a finite element approximation of this solution, $u_h \in V^{\polyorder}_{\mathcal{T}}$, when $\polyorder\geq 1$, and $q \geq \polyorder+1$ 
\begin{align}
    \left\| u- u_h \right\|_{L^{2}(\Omega)} \lesssim C_{I} h^{\polyorder+1}.
\end{align}
Here, $C_{I}$ is a constant that depends on the exact solution $u$ and its gradient.
\label{truncation_3d}
\end{assumption}

\begin{assumption}[Jump Identities]
    Consider a partial differential equation with an exact solution $u \in H^{q}(\Omega)$, where $q$ is a positive integer. The following jump identities hold for a finite element approximation of this solution, $u_h \in V^{\polyorder}_{\mathcal{T}}$, when $\polyorder \geq 1$, and $q \geq \polyorder+1$
    \begin{align}
        \left\| \llbracket u_{h} \rrbracket_F \right\|_{L_2(F)} & \lesssim C_{II} h^{\polyorder+1/2} , \\[1.5ex]
       \left\| \llbracket \frac{\partial u_h}{\partial e_{\perp ,1}} \rrbracket_{F} \right\|_{L_{2}(F)} & \lesssim C_{III} h^{\polyorder-1/2}, \\[1.5ex]
       \left\| \llbracket \frac{\partial u_h}{\partial e_{\perp ,2}} \rrbracket_{F} \right\|_{L_{2}(F)} & \lesssim C_{IV} h^{\polyorder-1/2}.
    \end{align}
    Here, we define $\llbracket \cdot \rrbracket$ such that
    \begin{align*}
        \llbracket u_h \rrbracket_{F} &= u_{h}^{+} - u_{h}^{-}, \qquad F \cap \partial \Omega = \emptyset, \\[1.0ex] \llbracket u_h \rrbracket_{F} &= u_h, \; \; \qquad \qquad F \cap \partial \Omega \neq \emptyset 
    \end{align*}
    is the jump in $u_h$, where $u_{h}^{+}$ and $u_{h}^{-}$ are the values of the solution on either side of interface $F$. Note that $F$ is a triangular face which belongs to the set of all faces of the triangulation $\mathcal{F}(\mathcal{T})$, equivalently $F \in \mathcal{F}(\mathcal{T})$. In addition,
    $C_{II}$, $C_{III}$, and $C_{IV}$ are constants that depend on the exact solution $u$ and its first derivatives. Furthermore, we note that $e_{\perp ,1}$ and $e_{\perp ,2}$ are two vectors that are orthogonal to each other and to an edge $e \in \partial F$.
\label{assumption_jump_3d}
\end{assumption}

Next, we define a smoothing operator (i.e.~a synchronization operator) which maps piecewise-polynomial functions into the WF space.

\begin{definition}[$H^2$ Smoothing Operator]
    Let us consider $E(\cdot): V^{\polyorder}_{\mathcal{T}} \rightarrow W_{\mathcal{T}}$. This is a smoothing and synchronization operator that has the following properties
    \begin{align}
        (E(u))(\vertex) &= \frac{1}{|\mathcal{T}^\vertex|}\sum_{T \in \mathcal{T}^\vertex} u_T(\vertex) \qquad \; \; \; \quad \forall \vertex \in \mathcal{V} (\mathcal{T}), \label{3d_point_k1} \\[1.5ex]
        \nabla (E(u))(\vertex) &= \frac{1}{|\mathcal{T}^\vertex|}\sum_{T \in \mathcal{T}^\vertex} \nabla u_T(\vertex) \quad \qquad \forall \vertex \in \mathcal{V} (\mathcal{T}),\\[1.5ex]
        \frac{\partial (E(u))}{\partial e_{\perp ,1}} (m) &= \frac{1}{|\mathcal{T}^m|}\sum_{T \in \mathcal{T}^m} \frac{\partial u_T}{\partial e_{\perp ,1}} (m) \qquad \forall m \in \mathcal{M} (\mathcal{T}),\\[1.5ex]
        \frac{\partial (E(u))}{\partial e_{\perp ,2}} (m) &= \frac{1}{|\mathcal{T}^m|}\sum_{T \in \mathcal{T}^m} \frac{\partial u_T}{\partial e_{\perp ,2}} (m) \qquad \forall m \in \mathcal{M} (\mathcal{T}).
    \end{align}
    Here, $\mathcal{T}^{\vertex} = \{T \in \mathcal{T} : \vertex \in \partial T \}$ is the set of tetrahedra that share a vertex $\vertex$, and $\mathcal{T}^m = \{T \in \mathcal{T} : m \in \partial T \}$ is the set of tetrahedra that share an edge midpoint~$m$.
    \label{definition_h2_3d}
\end{definition}

\begin{lemma}
    Consider a tetrahedron $T \in \mathcal{T}$, a vertex $\vertex$ of $T$, and an edge midpoint $m$ of $T$. Under these circumstances, the following expressions hold
    \begin{align}
        &|u_{T}(\vertex) - (E(u))(\vertex)|^{2} \lesssim \sum_{F \in \mathcal{F}^\vertex} |F|^{-1} \| \llbracket u \rrbracket_{F} \|^{2}_{L_{2}(F)} \quad \forall u \in V_{\mathcal{T}}^{\polyorder}, \label{3d_point_k1_bound}
    \end{align}
    \begin{align}
        \nonumber&|\nabla(u_{T})(\vertex) - \nabla(E(u))(\vertex) |^{2}\\[1.5ex] 
        &\lesssim \sum_{F \in \mathcal{F}^\vertex} \left( |F|^{-2} \| \llbracket u \rrbracket_F \|^{2}_{L_{2}(F)} + |F|^{-1} \left\| \llbracket \frac{\partial u}{\partial e_{\perp ,1}} \rrbracket_{F} \right\|^{2}_{L_{2}(F)} + |F|^{-1} \left\| \llbracket \frac{\partial u}{\partial e_{\perp ,2}} \rrbracket_{F} \right\|^{2}_{L_{2}(F)} \right) \quad \forall u \in V_{\mathcal{T}}^{\polyorder}, \label{3d_grad_k1_bound}
    \end{align}
    \begin{align}
        &\left| \frac{\partial u_{T}}{\partial e_{\perp ,1}}(m) - \frac{\partial (E(u))}{\partial e_{\perp ,1}}(m)  \right|^{2} \lesssim \sum_{F \in \mathcal{F}^m} |F|^{-1} \left\| \llbracket \frac{\partial u}{\partial e_{\perp ,1}} \rrbracket_{F} \right\|^{2}_{L_{2}(F)} \quad \forall u \in V_{\mathcal{T}}^{\polyorder}, \label{3d_norm1_k1_bound}
    \end{align}
    \begin{align}
        &\left| \frac{\partial u_{T}}{\partial e_{\perp ,2}}(m) - \frac{\partial (E(u))}{\partial e_{\perp ,2}}(m)  \right|^{2} \lesssim \sum_{F \in \mathcal{F}^m} |F|^{-1} \left\| \llbracket \frac{\partial u}{\partial e_{\perp ,2}} \rrbracket_{F} \right\|^{2}_{L_{2}(F)} \quad \forall u \in V_{\mathcal{T}}^{\polyorder}, \label{3d_norm2_k1_bound}
    \end{align}
    where $\mathcal{F}^\vertex$ is the set of interior faces sharing vertex $\vertex$, and $\mathcal{F}^m$ is the set of interior faces sharing midpoint $m$.
    \label{3d_point_wise_bounds}
\end{lemma}
\begin{proof}
    Consider two adjacent tetrahedra $T$ and $T' \in \mathcal{T}^{\vertex}$ that are face neighbors, and that share a common vertex $\vertex$.  The jump in the solution between the two tetrahedra can be expressed in terms of the jumps between (face) neighboring tetrahedra,  $\mathcal{T}^{\vertex} = \left\{ T_1, T_2, \ldots, T_j, T_{j+1}, \ldots, T_n \right\}$, $T_1 = T$, $T_n = T'$, that share the same vertex $\vertex$, as follows
    \begin{align}
        \nonumber |u_{T}(\vertex) - u_{T'}(\vertex)|^{2} &= \left| \sum^{n-1}_{j=1} \left[ u_{T_j}(\vertex) - u_{T_{j+1}}(\vertex) \right] \right|^{2}\\[1.5ex]
        \nonumber &\lesssim \sum^{n-1}_{j=1}|u_{T_j}(\vertex) - u_{T_{j+1}}(\vertex)|^2 = \sum^{n-1}_{j=1} \left| \llbracket u \rrbracket_{F_j} (\vertex) \right|^{2}\\[1.5ex]
        &\lesssim \sum^{n-1}_{j=1} |F|^{-1} \left\| \llbracket u \rrbracket_{F_j} \right\|^{2}_{L_{2}(F_j)}. \label{3d_point_jump}
    \end{align}
    Here, we have used the root-mean-square-arithmetic-mean  inequality in conjunction with the standard inverse estimate of~\cite{brenner2008mathematical}, Chapter 4.
    Upon combining Eqs.~\eqref{3d_point_k1} and~\eqref{3d_point_jump}, one obtains the first result (Eq.~\eqref{3d_point_k1_bound})
    \begin{align*}
        |u_{T}(\vertex) - (E(u))(\vertex)|^{2} = \left| \frac{1}{|\mathcal{T}^\vertex|} \sum_{T' \in \mathcal{T}^\vertex} \left( u_{T}(\vertex) - u_{T'}(\vertex) \right) \right|^{2} \lesssim \sum_{F \in \mathcal{F}^\vertex} |F|^{-1} \left\| \llbracket u \rrbracket_{F} \right\|^{2}_{L_{2}(F)}.
    \end{align*}
    Next, in a similar fashion, the following holds
    \begin{align}
        \nonumber&| \nabla (u_{T})(\vertex) - \nabla (E(u))(\vertex)  |^{2} \lesssim \sum_{F \in \mathcal{F}^\vertex} |F|^{-1} \| \llbracket \nabla u \rrbracket_F \|^{2}_{L_{2}(F)} \\[1.5ex]
        &= \sum_{F \in \mathcal{F}^{\vertex}} |F|^{-1} \left( \left\| \llbracket \frac{\partial u}{\partial e_{\perp ,1}} \rrbracket_{F} \right\|^{2}_{L_{2}(F)} + \left\| \llbracket \frac{\partial u}{\partial e_{\perp ,2}} \rrbracket_{F} \right\|^{2}_{L_{2}(F)} + \left\| \llbracket \frac{\partial u}{\partial t} \rrbracket_{F} \right\|^{2}_{L_{2}(F)} \right), \label{gradient_exp_3d}
    \end{align}
    where $\partial u / \partial t$ is the derivative along an edge in the tangential direction. In addition, the following inverse estimate holds
    \begin{align}
        \left\| \llbracket \frac{\partial u}{\partial t} \rrbracket_{F} \right\|^{2}_{L_{2}(F)} \leq |F|^{-1} \| \llbracket u \rrbracket_F \|^{2}_{L_{2}(F)}. \label{gradient_est_3d}
    \end{align}
    Upon substituting Eq.~\eqref{gradient_est_3d} into Eq.~\eqref{gradient_exp_3d}, we obtain the second desired result (Eq.~\eqref{3d_grad_k1_bound}).
    
    The final desired results (Eqs.~\eqref{3d_norm1_k1_bound} and~\eqref{3d_norm2_k1_bound}) immediately follow from standard inverse estimates of~\cite{brenner2008mathematical}, Chapter 4.
\end{proof}

\begin{lemma}[Bound on $H^2$ Smoothing Operator]
    Suppose $u \in V^{\polyorder}_{\mathcal{T}}$, where $1 \leq \polyorder \leq 3$. In addition, consider the smoothing operator $E(\cdot)$ which acts on $u$ and is defined in accordance with Definition~\ref{definition_h2_3d}. Under these circumstances, the following results hold
    \begin{align}
        &\nonumber \| u - E(u) \|^{2}_{L_2(T)}\\[1.5ex] 
        &\nonumber\lesssim \sum_{\vertex \in \mathcal{V}(T)} \sum_{F \in \mathcal{F}^\vertex} h \| \llbracket u \rrbracket_{F} \|^{2}_{L_{2}(F)} + \sum_{\vertex \in \mathcal{V}(T)} \sum_{F \in \mathcal{F}^\vertex} h^{3} \left( \left\| \llbracket \frac{\partial u}{\partial e_{\perp ,1}} \rrbracket_{F} \right\|^{2}_{L_{2}(F)} + \left\| \llbracket \frac{\partial u}{\partial e_{\perp ,2}} \rrbracket_{F} \right\|^{2}_{L_{2}(F)} \right) \\[1.5ex]
        &+ \sum_{m \in \mathcal{M}(T)} \sum_{F \in \mathcal{F}^m} h^{3} \left( \left\| \llbracket \frac{\partial u}{\partial e_{\perp ,1}} \rrbracket_{F} \right\|^{2}_{L_{2}(F)} + \left\| \llbracket \frac{\partial u}{\partial e_{\perp ,2}} \rrbracket_{F} \right\|^{2}_{L_{2}(F)} \right),
    \end{align}
    and similarly,
    \begin{align}
        &\nonumber \| u - E(u) \|^{2}_{L_2(\Omega)}\\[1.5ex] 
        &\lesssim \sum_{F \in \mathcal{F}(\mathcal{T})} h \| \llbracket u \rrbracket_{F} \|^{2}_{L_{2}(F)} + \sum_{F \in \mathcal{F}(\mathcal{T})} h^{3} \left( \left\| \llbracket \frac{\partial u}{\partial e_{\perp ,1}} \rrbracket_{F} \right\|^{2}_{L_{2}(F)} + \left\| \llbracket \frac{\partial u}{\partial e_{\perp ,2}} \rrbracket_{F} \right\|^{2}_{L_{2}(F)} \right),
    \end{align}
    where
    \begin{align*}
    h = \max_{T \in \mathcal{T}} h_{T}.
    \end{align*}
    \label{lemma_h2_3d}
\end{lemma}
\begin{proof} 
    The following scaling argument holds for all $w \in \mathcal{WF}(T)$
    \begin{align}
        \nonumber \| w \|^{2}_{L_{2}(T)} &\approx h^{3}_{T} \Bigg[ \sum_{\vertex \in \mathcal{V}(T)} \left( |w(\vertex)|^{2} + h^{2}_{T} | \nabla w(\vertex)|^{2} \right)  \\[1.5ex]
        &+ h^{2}_{T} \sum_{m \in \mathcal{M}(T)} \left( \left| \frac{\partial w}{\partial e_{\perp ,1}}(m)\right|^{2} + \left| \frac{\partial w}{\partial e_{\perp ,2}}(m)\right|^{2} \right) \Bigg]. \label{scaling_arg}
    \end{align}
    Next, we observe that $u - E(u) \in \mathcal{WF}(T)$ for all $T$. Therefore, upon setting $w = u - E(u)$ in Eq.~\eqref{scaling_arg}, we obtain
    \begin{align}
        \nonumber &\| u - E(u) \|^{2}_{L_{2}(T)}\\[1.5ex]
        \nonumber &\lesssim h^{3}_{T} \sum_{\vertex \in \mathcal{V}(T)} |u_{T}(\vertex) - (E(u))(\vertex)|^{2} + h^{5}_{T} \sum_{\vertex \in \mathcal{V}(T)} |\nabla u_{T}(\vertex) - \nabla (E(u))(\vertex)|^{2}\\[1.5ex]
        \nonumber &+ h^{5}_{T} \sum_{m \in \mathcal{M}(T)} \left| \frac{\partial u_{T}}{\partial e_{\perp ,1}}(m) - \frac{\partial (E(u))}{\partial e_{\perp ,1}}(m) \right|^{2} + h^{5}_{T} \sum_{m \in \mathcal{M}(T)} \left| \frac{\partial u_{T}}{\partial e_{\perp ,2}}(m) - \frac{\partial (E(u))}{\partial e_{\perp ,2}}(m) \right|^{2}\\[1.5ex]
        \nonumber &\lesssim \sum_{\vertex \in \mathcal{V}(T)} \sum_{F \in \mathcal{F}^\vertex} h_{T} \| \llbracket u \rrbracket_{F} \|^{2}_{L_{2}(F)} + \sum_{\vertex \in \mathcal{V}(T)} \sum_{F \in \mathcal{F}^\vertex} h^{3}_{T} \left( \left\| \llbracket \frac{\partial u}{\partial e_{\perp ,1}} \rrbracket_{F} \right\|^{2}_{L_{2}(F)} + \left\| \llbracket \frac{\partial u}{\partial e_{\perp ,2}} \rrbracket_{F} \right\|^{2}_{L_{2}(F)} \right) \\[1.5ex]
        &+ \sum_{m \in \mathcal{M}(T)} \sum_{F \in \mathcal{F}^m} h^{3}_{T} \left( \left\| \llbracket \frac{\partial u}{\partial e_{\perp ,1}} \rrbracket_{F} \right\|^{2}_{L_{2}(F)} + \left\| \llbracket \frac{\partial u}{\partial e_{\perp ,2}} \rrbracket_{F} \right\|^{2}_{L_{2}(F)} \right). \label{l2_bound_3d}
    \end{align}
    Here, we have used Eqs.~\eqref{3d_point_k1_bound}---\eqref{3d_norm2_k1_bound} from Lemma~\ref{3d_point_wise_bounds}, and set
    \begin{align}
        |F| \approx h_{T}^{2},
    \end{align}
   above. The final results is obtained by setting $h_{T} = h$ in Eq.~\eqref{l2_bound_3d}.
\end{proof}

\begin{remark}
    In Lemma~\ref{lemma_h2_3d}, we insist that the polynomial order $\polyorder \leq 3$. This is necessary in order to avoid pathological cases in which $u \in V^{\polyorder}_{\mathcal{T}}$ resides in the kernel of $E(\cdot)$. In particular, we note that bubble functions of degree $\polyorder = 4$ can reside in the kernel. More precisely, a generic $\polyorder = 4$ bubble function has $(1/6)(\polyorder+1)(\polyorder+2)(\polyorder+3) = 35$ degrees of freedom. Unfortunately, the operator $E(\cdot)$ is dependent on only 28 degrees of freedom (function values and derivative values) on the boundary of each element $T \in \mathcal{T}$. Therefore, if $\polyorder \geq 4$, it is possible for $E(u)$ to vanish, while some interior degrees of freedom for $u$ are non-zero. Under these circumstances, the inequalities in  Lemma~\ref{lemma_h2_3d} will not generally hold, as the left hand side can assume positive values, while the right hand side vanishes.
\end{remark}

\begin{definition}[Projection Operator] Consider a projection operator,
    $I_{h_b}(\cdot): L_{2}(\Omega) \rightarrow V_{\mathcal{T}_b}^{\polyorder}$ which transfers an $L_2$-function defined on mesh $\mathcal{T}_a$ to mesh $\mathcal{T}_b$.
    \label{proj_def}
\end{definition}

\begin{assumption}[Bound on Projection Operator] Let us define a function $w \in L_{2}(\Omega)$. The projection operator $I_{h_b}$ from Definition~\ref{proj_def} is bounded such that
\begin{align*}
     \left\| I_{h_b}(w) \right\|_{L_2(\Omega)} \lesssim C_{V} \left\| w \right\|_{L_2(\Omega)}.
\end{align*}
Here, $C_{V}$ is a constant independent of $w$.
\label{proj_assume}
\end{assumption}

\begin{lemma}[Projection Operator Error] Let us define a function $w \in H^{q}(\Omega)$. Under these circumstances, we have the following error estimate for the projection operator $I_{h_b}$ from Definition~\ref{proj_def} 
\begin{align*}
    \left\| w - I_{h_b}(w) \right\|_{L_2(\Omega)} \lesssim h_{b}^{\mu} |w|_{H^{q}(\Omega)}, 
\end{align*}
where $\mu = \min(\polyorder+1, q)$.  
\label{lemma_interp_bound_3d_k1}
\end{lemma}

\begin{proof}
    This result is well known. One may consult~\cite{stogner2007c1}, Section 2.3, as an example.
\end{proof}

\subsection{Key Results}

In what follows, we will combine the previous results, and construct a theorem which governs the order of accuracy for the smoothing and projection process.

\begin{theorem}[Solution Transfer Error Estimate]
Let us suppose that  Assumptions~\ref{truncation_3d}, \ref{assumption_jump_3d}, and \ref{proj_assume} hold, and $u \in H^{q}(\Omega)$ is the exact solution of our partial differential equation. Under these circumstances, the error between the solution on the initial mesh, $u_{h_a}$, and the transferred solution on the final mesh, $I_{h_b}(E(u_{h_a}))$, is governed by the following inequality
\begin{align}
    \left\| u_{h_a} - I_{h_b}(E(u_{h_a})) \right\|_{L_2(\Omega)} \lesssim \mathcal{C}_{u} h_a^{\polyorder+1}  + h^{\mu}_b |u|_{H^{q}(\Omega)}. 
    \label{solution_error_estimate}
\end{align}
Here, $\mathcal{C}_u$ is a constant that depends on $u$ and its gradient, and which is independent of $h_{a}$ and $h_{b}$. Furthermore, the exponent $\mu = \min(\polyorder+1, q)$ and we require that $1 \leq \polyorder \leq 3$. 
\label{error_estimate_theorem}
\end{theorem}

\begin{proof}
Let us begin by using the Triangle inequality
\begin{align}
    \nonumber \left\| u_{h_a} - I_{h_b}(E(u_{h_a})) \right\|_{L_2(\Omega)} &= \left\| u_{h_a} - I_{h_b}(E(u_{h_a})) + u -u +I_{h_b}(u) - I_{h_b}(u) \right\|_{L_2(\Omega)} \\[1.5ex]
    &\leq \left\| u_{h_a} - u \right\|_{L_2(\Omega)} + \left\| u - I_{h_b}(u) \right\|_{L_2(\Omega)} + \left\| I_{h_b}(u) - I_{h_b}(E(u_{h_a})) \right\|_{L_2(\Omega)}.
    \label{total_eqn}
\end{align}
We can extract three terms from the right hand side of Eq.~\eqref{total_eqn}
\begin{align}
    &\left\| u_{h_a} - u \right\|_{L_2(\Omega)}, \label{term_one}\\[1.5ex]
    &\left\| u - I_{h_b}(u) \right\|_{L_2(\Omega)}, \label{term_two} \\[1.5ex]
    &\left\| I_{h_b}(u) - I_{h_b}(E(u_{h_a})) \right\|_{L_2(\Omega)}. \label{term_three}
\end{align}
The first term (Eq.~\eqref{term_one}) can be bounded above by Assumption~\ref{truncation_3d} 
\begin{align}
    \left\| u - u_{h_a}  \right\|_{L_2(\Omega)} \lesssim C_{I} h^{\polyorder+1}_{a}. \label{term_one_final}
\end{align}
In addition, the second term (Eq.~\eqref{term_two}) can be bounded above in accordance with Lemma~\ref{lemma_interp_bound_3d_k1}
\begin{align}
    \left\| u - I_{h_b}(u) \right\|_{L_2(\Omega)} \lesssim h_{b}^{\mu} |u|_{H^{q}(\Omega)}. \label{term_two_final}
\end{align}
Finally, the third term (Eq.~\eqref{term_three}) can be bounded above in accordance with Assumption~\ref{proj_assume} and the Triangle inequality
\begin{align}
    \nonumber \left\| I_{h_b}(u) - I_{h_b}(E(u_{h_a})) \right\|_{L_2(\Omega)} &\lesssim C_{V} \left\| u - E(u_{h_a})\right\|_{L_2(\Omega)} \\[1.5ex]
    \nonumber &= C_{V} \left\| u - E(u_{h_a}) + u_{h_a} - u_{h_a}\right\|_{L_2(\Omega)}\\[1.5ex]
    &\lesssim C_{V} \left( \left\| u  - u_{h_a}\right\|_{L_2(\Omega)} +  \left\| u_{h_a} - E(u_{h_a})\right\|_{L_2(\Omega)} \right). \label{third_term_exp}
\end{align}
The right hand side of Eq.~\eqref{third_term_exp} requires further attention. It can be simplified by using Assumption~\ref{truncation_3d} on the first term on the right hand side, and Lemma~\ref{lemma_h2_3d} on the second term on the right hand side
\begin{align}
    \left\| u - u_{h_a}  \right\|_{L_2(\Omega)} &\lesssim C_{I} h^{\polyorder+1}_{a}, \label{term_zero_exp} \\[1.5ex] \nonumber
     \| u_{h_a} - E(u_{h_a}) \|_{L_2(\Omega)} &\lesssim \sum_{F \in \mathcal{F}(\mathcal{T}_{a})} h_{a}^{1/2} \| \llbracket u_{h_a} \rrbracket_{F} \|_{L_{2}(F)} \\[1.5ex]
     &+ \sum_{F \in \mathcal{F}(\mathcal{T}_{a})} h^{3/2}_{a} \left( \left\| \llbracket \frac{\partial u_{h_a}}{\partial e_{\perp ,1}} \rrbracket_{F} \right\|_{L_{2}(F)} + \left\| \llbracket \frac{\partial u_{h_a}}{\partial e_{\perp ,2}} \rrbracket_{F} \right\|_{L_{2}(F)} \right).
    \label{term_one_exp}
\end{align}
Now, let us substitute Eqs.~\eqref{term_zero_exp} and \eqref{term_one_exp} into Eq.~\eqref{third_term_exp}
\begin{align}
    \nonumber &\left\| I_{h_b}(u) - I_{h_b}(E(u_{h_a})) \right\|_{L_2(\Omega)} \\[1.5ex] 
    \nonumber &\lesssim C_{V} \Bigg(C_{I} h^{\polyorder+1}_{a} + \sum_{F \in \mathcal{F}(\mathcal{T}_{a})} h_{a}^{1/2} \| \llbracket u_{h_a} \rrbracket_{F} \|_{L_{2}(F)} \\[1.5ex]
    &+ \sum_{F \in \mathcal{F}(\mathcal{T}_{a})} h^{3/2}_{a} \left( \left\| \llbracket \frac{\partial u_{h_a}}{\partial e_{\perp ,1}} \rrbracket_{F} \right\|_{L_{2}(F)} + \left\| \llbracket \frac{\partial u_{h_a}}{\partial e_{\perp ,2}} \rrbracket_{F} \right\|_{L_{2}(F)} \right)  \Bigg).
\end{align}
Next, we use Assumption~\ref{assumption_jump_3d} in order to obtain
\begin{align}
    \nonumber \left\| I_{h_b}(u) - I_{h_b}(E(u_{h_a})) \right\|_{L_2(\Omega)} &\lesssim C_{V} \left(C_{I} h^{\polyorder+1}_{a} + \sum_{F \in \mathcal{F}(\mathcal{T}_a)} h_a^{\polyorder+1} \left( C_{II} + C_{III} + C_{IV} \right) \right) \\[1.5ex]
    &\lesssim C_{V} \left(C_{I}  +  \sum_{F \in \mathcal{F} (\mathcal{T}_a)}  \left( C_{II} + C_{III} + C_{IV} \right) \right) h_a^{\polyorder+1}.\label{term_three_final}
\end{align}
Lastly, combining Eqs.~\eqref{total_eqn}, \eqref{term_one_final}, \eqref{term_two_final}, and \eqref{term_three_final}, yields
\begin{align}
    \left\| u_{h_a} - I_{h_b}(E(u_{h_a})) \right\|_{L_2(\Omega)} \lesssim \left[C_{I} + C_{V} \left(C_{I}  +  \sum_{F \in \mathcal{F} (\mathcal{T}_a)}  \left( C_{II} + C_{III} + C_{IV} \right) \right) \right] h_a^{\polyorder+1} + h_{b}^{\mu} |u|_{H^{q}(\Omega)}.
\end{align}
We can obtain the desired result (Eq.~\eqref{solution_error_estimate}) by setting $\mathcal{C}_u$ equal to the quantity in square brackets. 
\end{proof}

We conclude this section by constructing a discrete maximum principle for the spline-interpolated solution.

\begin{theorem}[Discrete Maximum Principle]
    The smoothed WF representation on each tetrahedron $T$ must satisfy the following maximum principle
    \begin{align}
        \min_{i,j,\ell,m} \left(c_{ij\ell m}^{\alpha}\right) \leq u(x)\Big|_{T^{\alpha}} \leq \max_{i,j,\ell,m} \left(c_{ij\ell m}^{\alpha}\right), \qquad \forall \alpha
    \end{align}
    where $T = T^{\alpha = 1}\cup \hspace{0.25mm}T^{\alpha = 2}\cup \hspace{0.25mm}T^{\alpha = 3}\cdots \cup T^{\alpha = 11}\cup \hspace{0.25mm}T^{\alpha = 12}$ is a non-overlapping partition of $T$, and upon restricting $u$ to each subtetrahedron, one obtains a cubic polynomial (see definitions~\ref{def:bbpoly} and \ref{def:B-form}).
    \label{max_theorem}
\end{theorem}

\begin{proof}
    The proof follows immediately by examining a single subtetrahedron $T^{\alpha}$, and by noting that the Bernstein polynomials are a non-negative partition of unity on this subtetrahedron. In other words
    \begin{align*}
        \sum_{ij\ell m}B_{ij\ell m}(\lambda^{\alpha}) = 1, \qquad B_{ij\ell m}(\lambda^{\alpha}) \geq 0, 
    \end{align*}
    implies that
    \begin{align*}
        &\min_{i,j,\ell ,m} \left(c_{ij\ell m}^{\alpha} \right) = \left(\min_{i,j,\ell ,m} \left(c_{ij\ell m}^{\alpha} \right)\right) \sum_{ij\ell m}B_{ij\ell m}(\lambda^{\alpha}) \leq  \sum_{ij\ell m} c_{ij\ell m}^{\alpha} B_{ij\ell m}(\lambda^{\alpha}) = u(x)\Big|_{T^{\alpha}}, \\[1.0ex]
         &\max_{i,j,\ell ,m} \left(c_{ij\ell m}^{\alpha} \right) = \left(\max_{i,j,\ell ,m} \left(c_{ij\ell m}^{\alpha} \right)\right) \sum_{ij\ell m}B_{ij\ell m}(\lambda^{\alpha}) \geq \sum_{ij\ell m} c_{ij\ell m}^{\alpha} B_{ij\ell m}(\lambda^{\alpha}) = u(x)\Big|_{T^{\alpha}}.
    \end{align*}    
    A similar proof appears in Section 2 of~\cite{lohmann2017flux}.
\end{proof}

\begin{remark}
    In accordance with Theorem~\ref{max_theorem}, one may precisely control the maximum or minimum values of the smoothed WF representation of the solution. In particular, if the Bernstein representation is used to implement the WF-spline basis, one may directly modify the Bernstein coefficients in order to ensure that the solution $u$ remains within the desired range.
\end{remark}

\section{Numerical Results}
\label{sec;Results}
In this section, we assess the WF solution-transfer process on the basis of order of accuracy and mass conservation. Additionally, we qualitatively assess the visualization properties of the WF transfer process. We have narrowed the  focus of these studies to compare the WF transfer with linear interpolation---which provides a standard of performance that we can measure our method against---for polynomial orders $\polyorder = 1$ and $\polyorder = 2$. In order to evaluate the performance of the solution transfer methods, we tested their ability to transfer between a series of unstructured tetrahedral \emph{source} meshes and a similarly sized series of unstructured tetrahedral \emph{target} meshes. The properties of both mesh series are summarized in Table~\ref{tab:grid_info}. 
\begin{table}[h!]
\begin{center}
\begin{tabular}{|p{1.5cm}|p{2cm}|p{2cm}|p{1.5cm}|}
\hline
Grid Sequence Number& Source \newline Elements& Target \newline Elements& $h$\\
\hline
1& 192&	194&	0.173340\\
\hline
2 & 1536&	1544&	0.086670\\
\hline
3 & 12288&	12388&	0.043335\\
\hline
4 & 98304&	98311&	0.021668\\
\hline
5 & 786432&	783998&	0.010834\\
\hline
\end{tabular}
\end{center}
\caption{Properties of the tetrahedral source and target meshes developed for the mass conservation and order of accuracy studies. Note $h$ is the estimated element size, and is calculated by $1/\sqrt[3]{\# ofelements}$.}
\label{tab:grid_info}
\end{table}

We selected two functions for the numerical studies presented in this work,
\begin{align}
    u_1 = \exp{\left(-30\left((x-0.5)^2+(y-0.5)^2+(z-0.5)^2\right)\right)}, \hspace{2mm} x,y,z \in [0,1],
    \label{eqn:u_1}
\end{align}
\begin{align}
    u_2 = &\tanh{\big(20\big((x-0.5)+0.3\sin{(-10(y-0.5))}-0.3\sin{(-5(z-0.6))}\big)\big)},\nonumber\\&x,y,z \in [0,1],
     \label{eqn:u_2}
\end{align}
which were taken from Alauzet~\cite{alauzet2016parallel} and shifted from $[-0.5,0.5]^3$ to $[0,1]^3$ to avoid any convenient error cancellations from being centered about the origin. The first function, $u_1$, given in Eq.~\ref{eqn:u_1}, is a Gaussian function that replicates the vortices found in CFD simulation~\cite{alauzet2016parallel}. The second function, $u_2$, given in Eq.~\ref{eqn:u_2}, is a hyperbolic tangent function, which while remaining smooth, features a significant gradient. Both of these functions are visualized on the top left of Figures~\ref{fig:gauss_viz} through~\ref{fig:tanh_planes_viz}. For the experiments, we used an $L_2$-projection to project $u_1$ and $u_2$ onto the source grid, thus producing a piecewise polynomial, discontinuous function on the source mesh. The purpose of this function was to mimic the numerical solution provided by a discontinuous finite element method.  

\subsection{Order of Accuracy}
\label{sec:Order of Accuracy}
\noindent For the order of accuracy study, the $L_2$-error is computed as follows
\begin{align}
     \varepsilon_{L_2}= \sqrt{\frac{1}{|\Omega|}\int_{\Omega}\left(u-g\right)^2d\Omega}.
     \label{eqn:L2_error}
\end{align}
Note that in this context, $u$ refers, generally, to either of the projected, discontinuous functions, $u_1$ or $u_2$, on the source mesh. The transferred solution on the target mesh is referred to as $g$. The integral in this expression was computed over the whole domain $\Omega$, using the tensor product of three, 41-point Gauss quadrature rules. This yielded a sufficiently accurate approximation of the integral based on the error values at 68,921 quadrature points. The rationale for this method of measuring error is due to our limited options of how such integration can be performed. Since there are only two meshes available, we can compute the L$_2$-error of the solution on the source mesh, on the target mesh, or on a rectilinear grid of points. We chose the latter, because the error function is defined based on solution functions on two different meshes, and the natural location of integration is ambiguous. In addition, we wanted to maintain a level of consistency by evaluating our error at the same points across multiple grids. We suspect that if the global quadrature were a problem, we would see consistent, systematic errors across multiple cases and those are not present in our order of accuracy studies.

In this work we also look at the convergence of the gradient magnitude of the transferred solution for both the linear interpolation   and WF methods. The results of our $L_2$-error studies are shown for the Gaussian function $u_1$ and its gradient magnitude in Figure~\ref{fig:gauss_l2}. In addition, Figure~\ref{fig:tanh_L2} shows the error of the transferred solution and its gradient magnitude for the hyperbolic tangent function, $u_2$. 
\begin{figure}[h!]
\centering
    \begin{subfigure}{0.5\textwidth}
        \centering
        \includegraphics[width=1.0\textwidth]{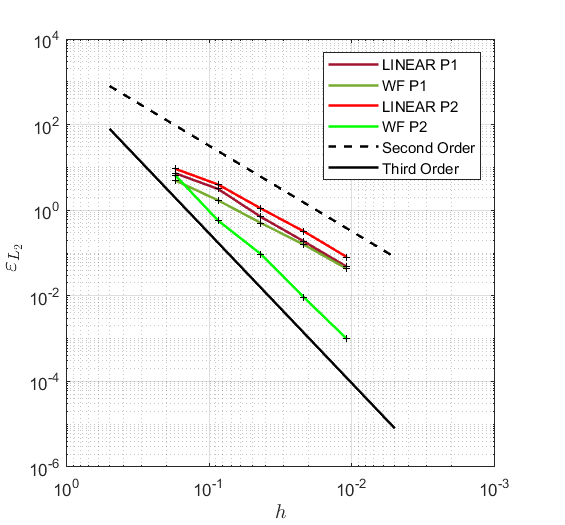}
    \end{subfigure}%
    \begin{subfigure}{0.5\textwidth}
        \centering
        \includegraphics[width=1.0\textwidth]{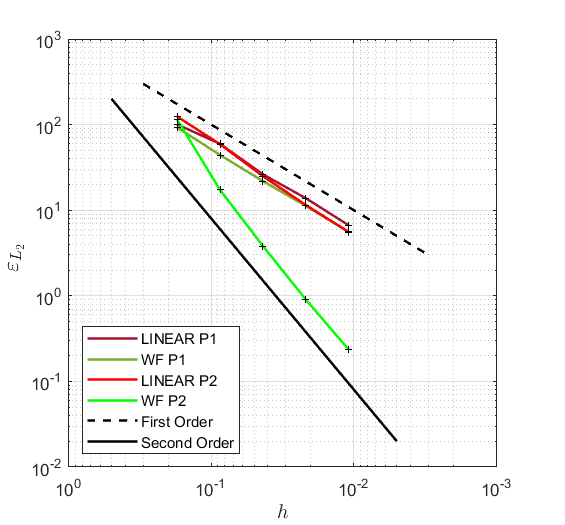}
    \end{subfigure}
    \caption{Left, the $L_2$-error of the Gaussian function $u_1$, and right, the $L_2$-error of the gradient magnitude of $u_1$, with linear interpolation for $\mathcal{P}_1$ and $\mathcal{P}_2$ data shown in shades of red and WF transfer for $\mathcal{P}_1$ and $\mathcal{P}_2$ data shown in shades of green.}
    \label{fig:gauss_l2}
\end{figure}

The rate of convergence for the WF transfer, shown in Figure~\ref{fig:gauss_l2}, meets our predicted/expected convergence rate of $\polyorder+1$ for both  $\polyorder = 1$ and $\polyorder=2$, i.e.~we see a second order convergence rate for $\polyorder=1$ and third order rate for $\polyorder=2$. We can also appreciate the clear separation in solution accuracy between the $\polyorder=2$ WF transfer and $\polyorder=2$ linear interpolation of the solution. This separation is even more apparent in the study of the gradient magnitude. For this study, it appears that the order of accuracy of the gradient magnitude converges at $\polyorder$, thus $\polyorder=2$ WF transfer sees second order convergence, and $\polyorder=1$ WF transfer sees first order convergence. Meanwhile, linear interpolation is first order convergent regardless of the order of the data being transferred. 
\begin{figure}[h!]
\centering
    \begin{subfigure}{0.5\textwidth}
        \centering
        \includegraphics[width=1.0\textwidth]{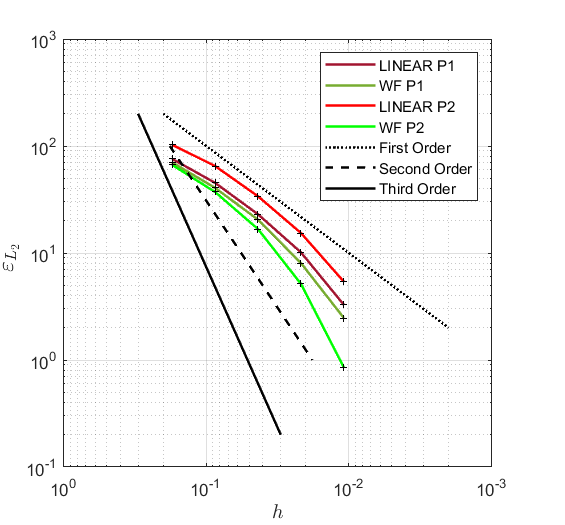}
    \end{subfigure}%
    \begin{subfigure}{0.5\textwidth}
        \centering
        \includegraphics[width=1.0\textwidth]{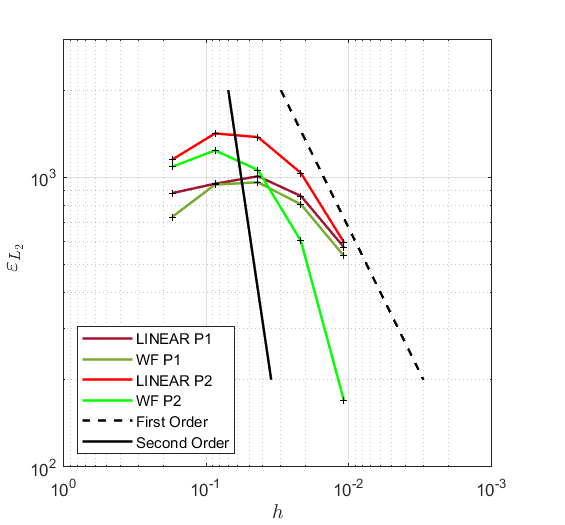}
    \end{subfigure}
    \caption{Left, the $L_2$-error of the hyperbolic tangent function $u_2$, and right, the $L_2$-error of the gradient magnitude of $u_2$, with linear interpolation for $\mathcal{P}_1$ and $\mathcal{P}_2$ data shown in shades of red and WF transfer for $\mathcal{P}_1$ and $\mathcal{P}_2$ data shown in shades of green.}
    \label{fig:tanh_L2}
\end{figure}

The $L_2$-error study for the hyperbolic tangent function, shown in Figure~\ref{fig:tanh_L2}, was hindered by an initial lack of resolution necessary to accurately resolve the features of the field, hence why there is an arc to the convergence for all methods tested. WF transfer for both $\polyorder=1$ and $\polyorder=2$ eventually achieves our expected convergence rate of $\polyorder+1$, but only on finer grids. For $\polyorder=1$ and $\polyorder=2$, WF transfer remains more accurate than linear interpolation for both the solution and its gradient magnitude, and the difference in accuracy is especially apparent for $\polyorder=2$. 

In summary, this study indicates that WF transfer is particularly well-suited for applications which require high fidelity, as it better preserves the accuracy of the solution and the gradient magnitude of the solution. We also confirmed the expected convergence rates for both $\polyorder=1$ and $\polyorder=2$ WF transfer, which were second and third order, respectively. We can deduce from this study, and our previous studies in 2D~\cite{larose2025spline}, that provided the function is smooth, the rate of convergence for the WF transfer will behave as theoretically predicted. Meanwhile, functions which contain sharp features will require sufficient mesh resolution before the expected convergence rates are realized. For this reason, and for the sake of brevity, we have limited our study of the order of accuracy to just two functions---one that represents each regime. 

\subsection{Mass Conservation} \label{sec:mass_conservation_tests}

We performed a study on the conservative nature of WF transfer, linear interpolation, and $L_2$-projection transfer for the Gaussian function in Figure~\ref{fig:mass_conservation}. We have included the $L_2$-projection transfer  in order to create additional data for comparison purposes. We note that the $L_2$-projection transfer process is simply the WF transfer process without the synchronization or spline interpolation steps. So, in this sense it offers a meaningful reference point.
\begin{figure}[h!]
    \centering
    \includegraphics[width=0.75\linewidth]{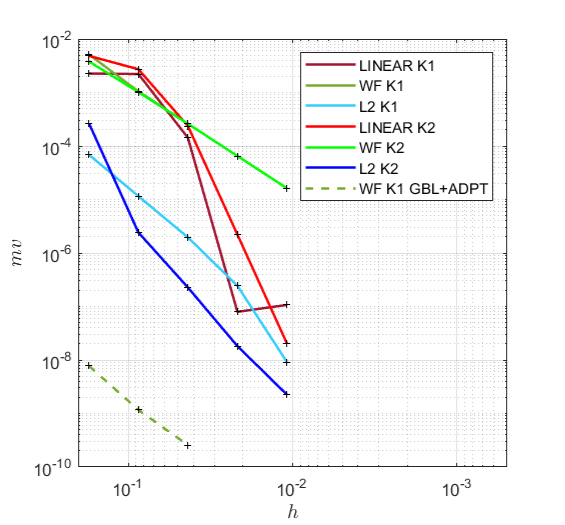}
    \caption{A comparison of the mass variation for the solution transfer methods: linear, WF, $L_2$-Projection, and WF $\polyorder=1$ with global $L_2$-projection and adaptive quadrature. Results for both $\polyorder=1$ and $\polyorder=2$ are shown for linear (shades of red), WF (shades of green), and $L_2$-Projection (shades of blue).}
    \label{fig:mass_conservation}
\end{figure}

In our conservation tests, all $L_2$-projections were computed using a 56-point quadrature rule for tetrahedra~\cite{shunn2012symmetric}, which is capable of exactly integrating a polynomial of degree 9.

In Figure~\ref{fig:mass_conservation}, we see that, despite its high-order accuracy, the WF solution transfer process is often less conservative than linear interpolation or the $L_2$-projection transfer process. This is not  completely surprising, as the WF transfer process is only guaranteed to be asymptotically conservative, not exactly conservative. As we mentioned previously, conservation errors in the WF transfer process can be introduced during any of its three steps: synchronization, spline-reconstruction, or $L_2$-projection. However, it was not immediately clear to us \emph{a priori} which step was responsible for the largest conservation errors.
Therefore, we investigated each step of the transfer process by integrating over each element in the domain after each step and comparing with our \emph{source} mass prior to any operations. From this investigation, we determined that virtually no error is introduced by synchronization, a small amount of error is introduced by the $L_2$-projection (we can see the magnitude of the error introduced by the $L_2$-projection from Figure~\ref{fig:mass_conservation}), and the vast majority of the error is introduced when reconstructing/interpolating the solution using the WF splines. In what follows, we discuss two optional strategies for reducing the latter two sources of conservation error. We note that, regardless of whether these strategies are employed, the original WF transfer method remains asymptotically conservative. 

\subsubsection{Reducing Projection Error} \label{projection_error}

In~\cite{larose2025spline}, we introduced the idea of using adaptive quadrature procedures to improve the accuracy of $L_2$-projections on simplex elements. We showed both the impact of $h$-adaptation (splitting the quadrature domain into smaller subdomains), $p$-adaptation (increasing the strength of the quadrature rule by increasing the number of quadrature points), and hybrid $hp$-adaptation strategies. However, during our previous study, the quadrature process was not fully adaptive: we used a fixed number of subdivisions or quadrature points based on preliminary testing, and these parameters were not adjusted based on local estimates of the quadrature error. For the present updated study, we recursively subdivided the elements, using an edge-midpoint approach like that presented in~\cite{liu1996quality}, and applied a 56-point quadrature rule to each subdivision until the difference between consecutive integration approximations fell below a predefined tolerance. This was accomplished by setting a tolerance on the integrals of our $L_2$-projection and recursively $h$-adapting (subdividing) each element until the change in the integrals met the specified tolerance. The process was concluded once the integration tolerance ($1.0\times10^{-12}$) was met, or a maximum  of four subdivisions was performed.

\subsubsection{Reducing Spline Interpolation Error} \label{spline_error}

The spline reconstruction/interpolation error can be reduced by performing a global $L_2$-projection of the discontinuous solution on the source grid to the spline representation on the source grid---effectively replacing the need for the synchronization process  and the geometric computation of coefficients (see the Appendix). This operation requires solving a linear system whose size is proportional to the following three factors: the number of macrotetrahedra in the mesh, the number of subtetrahedra per macrotetrahedra (12), and the number of $\mathcal{P}_3$ degrees of freedom (DoFs) associated with each subtetrahedra (20). Note that some DoFs are duplicated, as they are shared by multiple subtetrahedra. During this process, we still follow the standard WF procedure but now require a domain-wide global mapping that will select the appropriate coefficients for each subtetrahedra. For this purpose, we imported a WF-refined mesh and allowed \emph{GMSH} to endow each element with a global node numbering using $\mathcal{P}_3$ nodes. We then constructed a local map shown in the first and second columns of Table~\ref{tab:Mappings}, which we used to select the appropriate coefficients from the global data returned by \emph{GMSH}. Thereafter, the global $L_2$-projection system was assembled, and solved for the spline coefficients.

\subsubsection{Combined Results and Final Perspectives}
The combination of techniques from Sections~\ref{projection_error} and \ref{spline_error} greatly reduce the  conservation error of the WF transfer as shown by the green dotted line in Fig. ~\ref{fig:mass_conservation}. For example, on the coarsest set of grids, the total conservation error was reduced from $5.152754\times10^{-3}$ to $7.758091\times10^{-9}$ for the Gaussian function (see Table~\ref{tab:mass_results_grid1to3}). We see that the effects of our new techniques shift the conservation error by roughly six orders of magnitude, thus giving us faster asymptotic convergence with significantly less resolution required. 

\begin{table}[h!]
    \begin{center}
    \begin{tabular}{|p{1cm}|p{1.8cm}|p{2.5cm}|p{2.5cm}|p{2.5cm}|}
    \hline
    \multicolumn{5}{|c|}{ Gaussian WF Transfer}\\
    \hline
    Grid & Step & Geometric ($\mathcal{C}^1$) & Geometric ($\mathcal{C}^1$) \& Adaptive Quad. & Global Coeff. \& Adaptive Quad. ($\mathcal{C}^0$)\\
    \hline
    \hline
    \multirow{4}{5em}{Grid 1} &Sync. & $4.163336\times10^{-17}$ & $4.163336\times10^{-17}$ & N/A \\
    \cline{2-5}
    &Spline-Rep. & $5.200065\times10^{-3}$ & $5.200065\times10^{-3}$ & $1.387779\times10^{-17}$\\
    \cline{2-5}
    &$L_2$-proj. & $4.731106\times10^{-5}$ & $1.372977\times10^{-10}$ &  $7.758091\times10^{-9}$\\
    \cline{2-5}
    &Total & $5.152754\times10^{-3}$ & $5.200065\times10^{-3}$  & $7.758091\times10^{-9}$\\
    \hline
    \hline
    \multirow{4}{5em}{Grid 2}& Sync. & $7.632783\times10^{-17}$ & $7.632783\times10^{-17}$ & N/A \\
    \cline{2-5}
    &Spline-Rep. & $1.046923\times10^{-3}$ & $1.046923\times10^{-3}$ & $8.049117\times10^{-16}$\\
    \cline{2-5}
    &$L_2$-proj. & $6.024885\times10^{-7}$ & $1.492484\times10^{-10}$ &  $1.167108\times10^{-9}$\\
    \cline{2-5}
    &Total & $1.047525\times10^{-3}$ & $1.046923\times10^{-3}$  & $1.167107\times10^{-9}$\\
    \hline
    \hline
    \multirow{4}{5em}{Grid 3} & Sync. & $5.551115\times10^{-17}$ & $5.551115\times10^{-17}$ & N/A \\
    \cline{2-5}
    & Spline-Rep. & $2.597638\times10^{-4}$ & $2.597638\times10^{-4}$ & $9.381385\times10^{-15}$\\
    \cline{2-5}
    & $L_2$-proj. & $1.993326\times10^{-8}$ & $3.395055\times10^{-12}$ &  $2.567850\times10^{-10}$\\
    \cline{2-5}
    & Total & $2.597439\times10^{-4}$ & $2.597638\times10^{-4}$  & $2.567756\times10^{-10}$\\
    \hline
    \end{tabular}
    \end{center}
    \caption{A comparison of the mass error for the standard WF transfer, WF transfer with adaptive quadrature, and WF transfer with global $L_2$-projection and adaptive quadrature on grids 1-3.}
    \label{tab:mass_results_grid1to3}
\end{table}

Of course, there are trade-offs to performing a global $L_2$-projection, as it will reduce the continuity of the solution from $\mathcal{C}^1$ to $\mathcal{C}^0$, which will result in diminished visualization and solution gradient accuracy. It may also require solving a large linear system for the coefficients. We also note that the proposed adaptive quadrature procedure can become expensive, especially if machine-zero accuracy of the projection is required.

It is our recommendation to consider the modularity of this approach: one can use the conservation properties of the $\mathcal{C}^0$ global $L_2$-projection to run and perform calculations, and use the geometric $\mathcal{C}^1$ WF reconstruction to visualize the solution and its gradient more accurately. One may also adjust the maximum number of refinement levels in the adaptive quadrature process in order to manage computational cost. One's choices depend on the needs of one's particular application.

\begin{table}[h!]
    \begin{center}
    \begin{tabular}{|c|c|c|c|c|c|c|c|c|c|c|c|c|c|}
    \hline
    \multicolumn{14}{|c|}{Element Mappings}\\
    \hline
    $\xi$ & GMSH & $T^{\alpha = 1}$ & $T^{2}$ & $T^{3}$ & $T^{4}$ & $T^{5}$ & $T^{6}$ & $T^{7}$ & $T^{8}$ & $T^{ 9}$ & $T^{10}$ & $T^{ 11}$ & $T^{12}$\\
    \hline
    3000 & 1 & 1 & 2 & 3 & 2 & 4 & 3 & 4 & 1 & 3 & 2 & 1 & 4\\
    \hline
    0300 & 2 & 2 & 3 & 1 & 4 & 3 & 2 & 1 & 3 & 4 & 1 & 4 & 2\\
    \hline
    0030 & 3 & 5 & 5 & 5 & 12 & 12 & 12 & 19 & 19 & 19 & 26 & 26 & 26\\
    \hline
    0003 & 4 & 91 & 91 & 91 & 91 & 91 & 91 & 91 & 91 & 91 & 91 & 91 & 91\\
    \hline
    2100 & 5 & 33 & 35 & 39 & 41 & 38 & 36 & 43 & 40 & 37 & 34 & 44 & 42\\
    \hline
    1200 & 6 & 34 & 36 & 40 & 42 & 37 &35 & 44 & 39 & 38 & 33 & 43 & 41\\
    \hline
    0210 & 7 & 9 & 10 & 11 & 16 & 17 & 18 & 23 & 24 & 25 & 31 & 32 & 30\\
    \hline
    0120 & 8 & 6 & 7 & 8 & 13 & 14 & 15 & 20 & 21 & 22 & 28 & 29 & 27\\
    \hline
    1020 & 9 & 8 & 6 & 7 & 15 & 13 & 14 & 22 & 20 & 21 & 27 & 28 & 29\\
    \hline
    2010 & 10 & 11 & 9 & 10 & 18 & 16 & 17 & 25 & 23 & 24 & 30 & 31 & 32\\
    \hline
    1002 & 11 & 47 & 45 & 46 & 45 & 51 & 46 & 51 & 47 & 46 & 45 & 47 & 51\\
    \hline
    2001 & 12 & 50 & 48 & 49 & 48 & 52 & 49 & 52 & 50 & 49 & 48 & 50 & 52\\
    \hline
    0012 & 13 & 53 & 53 & 53 & 55 & 55 & 55 & 57 & 57 & 57 & 59 & 59 & 59\\
    \hline
    0021 & 14 & 54 & 54 & 54 & 56 & 56 & 56 & 58 & 58 & 58 & 60 & 60 & 60\\
    \hline
    0102 & 15 & 45 & 46 & 47 & 51 & 46 & 45 & 47 & 46 & 51 & 47 & 51 & 45\\
    \hline
    0201 & 16 & 48 & 49 & 50 & 52 & 49 & 48 & 50 & 49 & 52 & 50 & 52 & 48\\
    \hline
    1110 & 17 & 61 & 62 & 63 & 64 & 65 & 66 & 67 & 68 & 69 & 71 & 72 & 70\\
    \hline
    1101 & 18 & 73 & 77 & 78 & 80 & 79 & 77 & 87 & 78 & 79 & 73 & 87 & 80\\
    \hline
    1011 & 19 & 76 & 74 & 75 & 81 & 83 & 82 & 84 & 85 & 86 & 89 & 88 & 90\\
    \hline
    0111 & 20 & 74 & 75 & 76 & 83 & 82 & 81 & 85 & 86 & 84 & 88 & 90 & 89\\
    \hline
    \end{tabular}
    \end{center}
    \caption{Mapping of the B-coefficients from their local domain ordering, to the order in which \emph{GMSH} outputs $\mathcal{P}_3$ nodes, for the subtetrahedra in a macroelement.}
    \label{tab:Mappings}
\end{table}

\subsection{Visualization}
\label{sec:Visualization}

In our previous work~\cite{larose2025spline}, we showed that spline-based solution  transfer greatly improves visualization of the solution in 2D. In this section, we show that this desirable property extends to 3D. 

In our study of the visualization properties, we present comparisons of smoothed and un-smoothed data for the functions $u_1$ and $u_2$.  In each case, the un-smoothed data was created by computing the $L_2$-projection of an analytical function on source mesh number 2. The polynomial degree for this projection was $\polyorder = 1$. Thereafter, the smoothed data was generated by synchronization and interpolation with WF splines (of degree 3).  In order to visualize the un-smoothed solution, we sampled it onto a rectilinear grid of $33\times33\times33$ Gauss  quadrature points, for a total of 35,937 sampling points. In addition, in order to more accurately represent the cubic-smoothed data, we sampled it onto a rectilinear grid of  $97\times97\times97$ Gauss quadrature points, for a total of 912,673 sampling points. Lastly, the original analytical function, i.e.~the exact solution, was sampled on the latter rectilinear grid. After sampling was completed, visualization was performed using the open-source, post-processing visualization tool, \emph{ParaView}.

Figure~\ref{fig:gauss_viz} shows the Gaussian function $u_1$ and its gradient magnitude. The exact solution appears in the top row, the $L_2$-projection appears in the middle row, and the WF interpolation appears on the bottom row. What is most striking from this figure is that the WF gradient magnitude (see the right side of each row) better resolves the inner, dark-red ring thus capturing the gradient far better than the $L_2$-projection---which echoes what we saw from the order of accuracy study. The solution (see the left side of each row) has little noticeable variation across all rows, though we know thanks to our order of accuracy study, that the WF process is slightly more accurate, even for $\polyorder=1$ data. 
\begin{figure}
    \centering
    \includegraphics[width=1.0\linewidth]{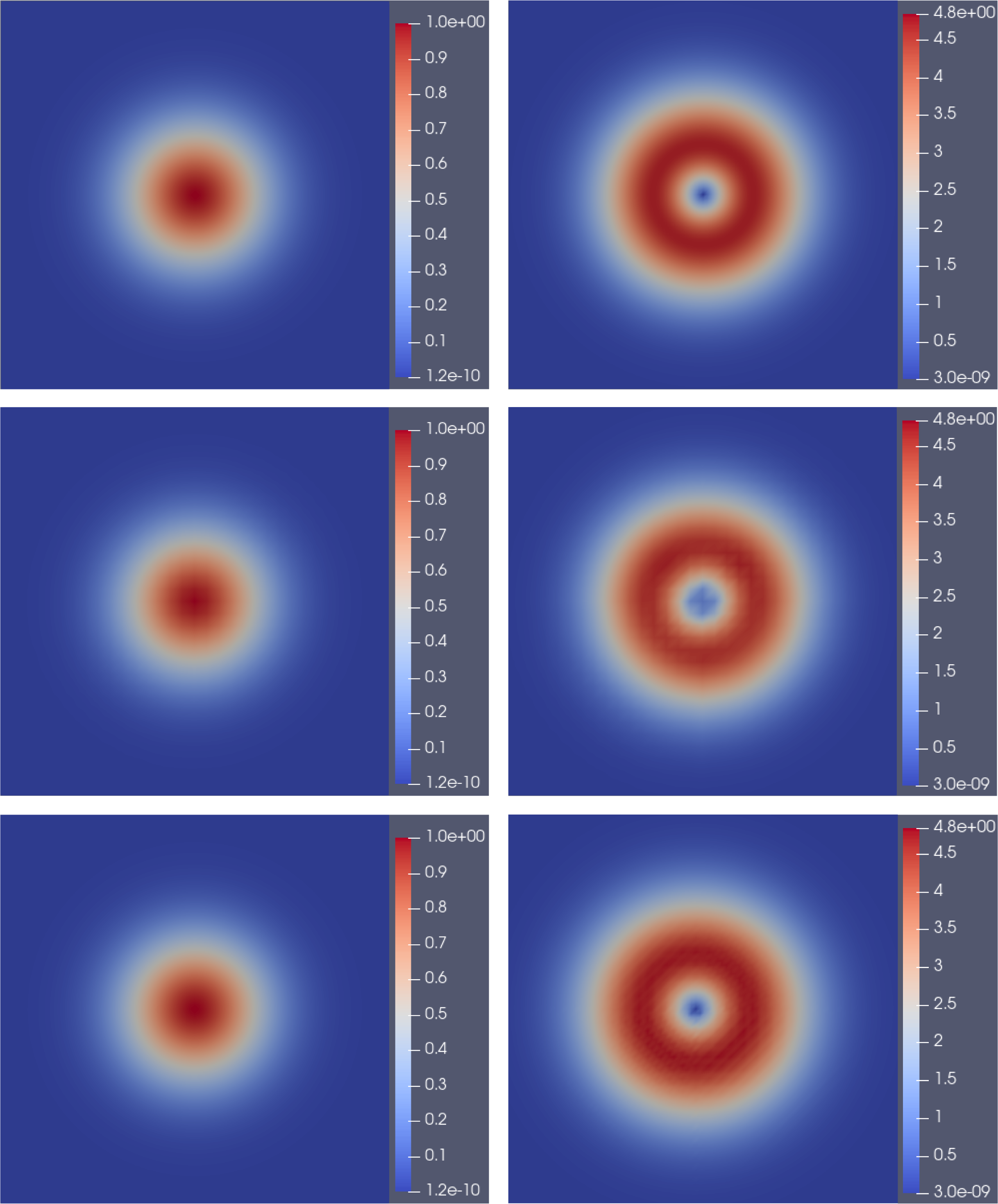}
    \caption{Visualization of the Gaussian function, $u_1$, and its gradient magnitude over the plane $x=0.5$, $y,z\in[0,1]$. This data is presented as follows: top left, the exact solution $u_1$; top right, the gradient magnitude of the exact solution $u_1$; middle left, the $L_2$-projection of $u_1$; middle right, the gradient magnitude of the $L_2$-projection of $u_1$; bottom left, the WF interpolant of $u_1$; and bottom right, the gradient magnitude of the WF interpolant of $u_1$.}
    \label{fig:gauss_viz}
\end{figure}

Figures~\ref{fig:tanh_full_viz} and \ref{fig:tanh_planes_viz} show the hyperbolic tangent function $u_2$ and its gradient magnitude.  Figure~\ref{fig:tanh_planes_viz} very clearly illustrates the improvement of the WF process over $L_2$-projection. Not only does the gradient magnitude look smoother and better resolved, but the solution is also noticeably better resolved along the sharp, diamond-shaped boundary between the red and blue regions.

\begin{figure}
    \centering
    \includegraphics[width=1.0\linewidth]{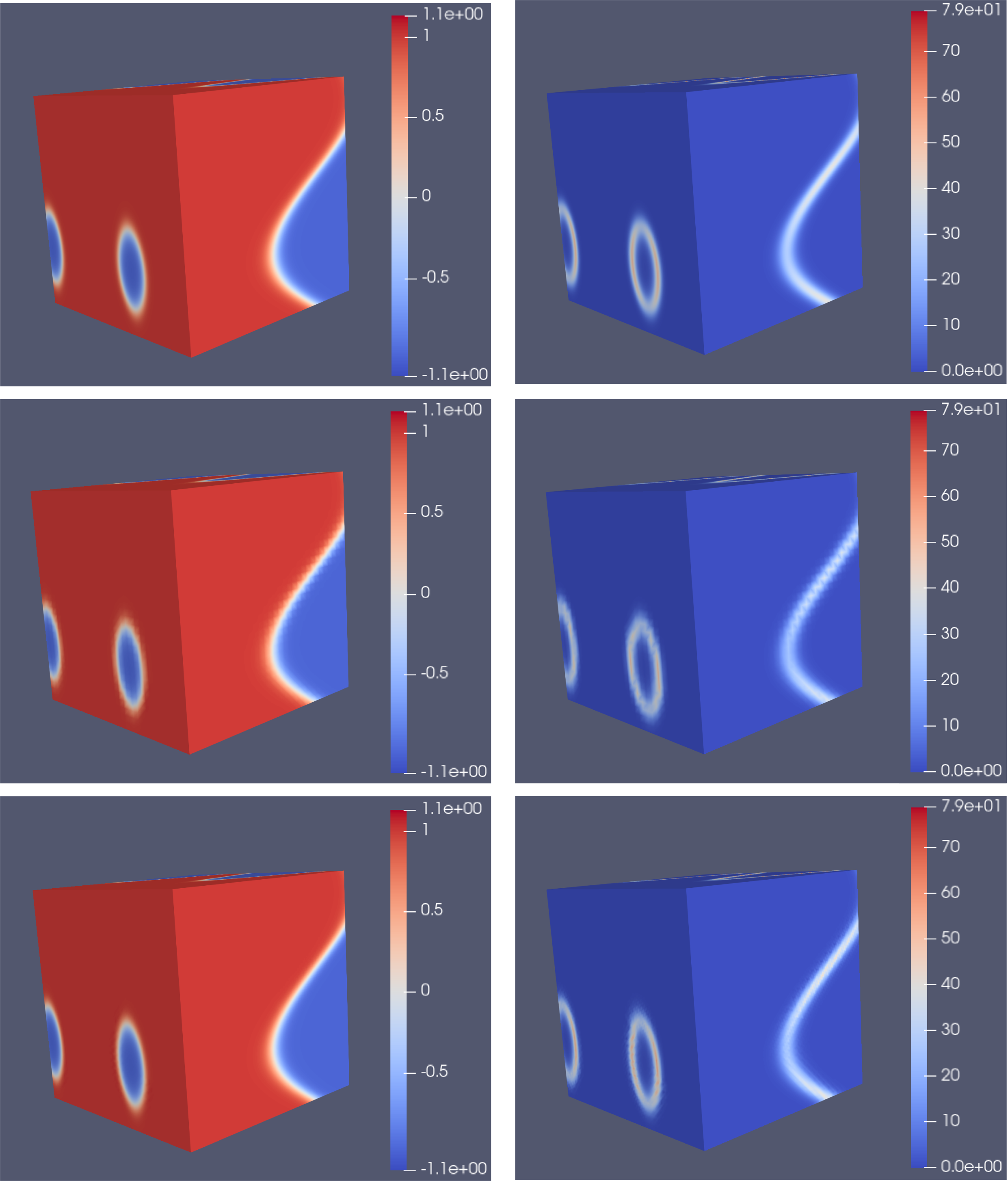}
    \caption{Visualization of the hyperbolic tangent function, $u_2$, and its gradient magnitude over the whole domain $x,y,z \in [0,1]$. This data is presented as follows: top left, the exact solution $u_2$; top right, the gradient magnitude of the exact solution $u_2$; middle left, the $L_2$-projection of $u_2$; middle right, the gradient magnitude of the $L_2$-projection of $u_2$; bottom left, the WF interpolant of $u_2$; and bottom right, the gradient magnitude of the WF interpolant of $u_2$.}
    \label{fig:tanh_full_viz}
\end{figure}
\begin{figure}
    \centering
    \includegraphics[width=1.0\linewidth]{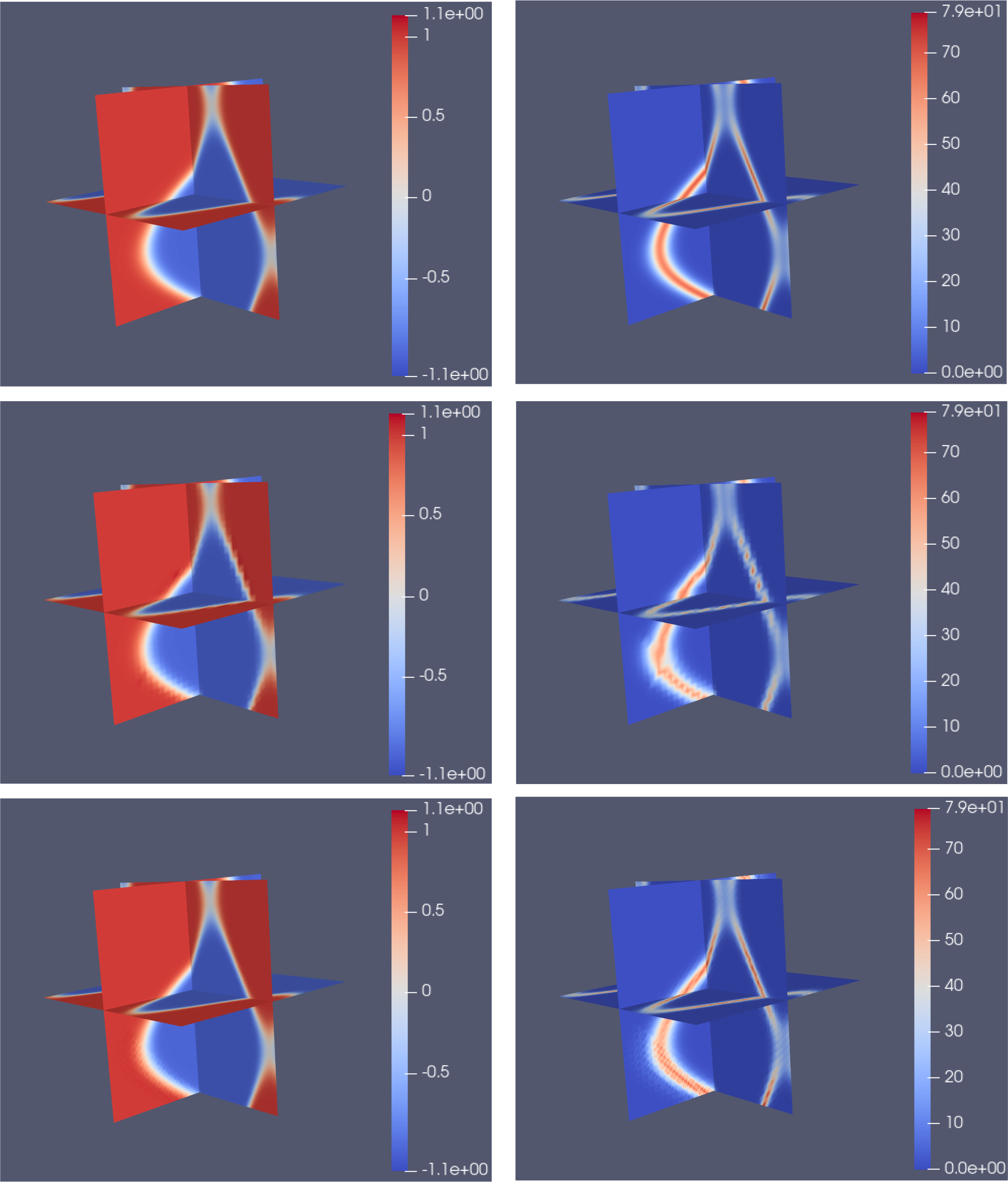}
    \caption{Visualization of the hyperbolic tangent function, $u_2$, and its gradient magnitude on planes $\{x=0.5,\, y,z \in [0,1]\}$, $\{y=0.5,\,x,z=[0,1]\}$, $\{z=0.5,\,x,y=[0,1]\}$. This data is presented as follows: top left, the exact solution $u_2$; top right, the gradient magnitude of the exact solution $u_2$; middle left, the $L_2$-projection of $u_2$; middle right, the gradient magnitude of the $L_2$-projection of $u_2$; bottom left, the WF interpolant of $u_2$; and bottom right, the gradient magnitude of the WF interpolant of $u_2$.}
    \label{fig:tanh_planes_viz}
\end{figure}

The results above are unsurprising, as the gradient (and its magnitude) are discontinuous for the $L_2$-projection, whereas they are continuous for the WF approach. However, we note that the discontinuities of the gradient magnitude for the $L_2$ projection are not easy to directly observe. This is due to the inherent limitations of modern visualization software. In 2D, directly visualizing the discontinuities is actually possible, as the gradient magnitude can be plotted in a 3D format: i.e.~a surface can be plotted with coordinates $x$ and $y$ that correspond to the physical locations of points in the spatial domain, along with a coordinate $z$ (the height of the surface) that corresponds to the gradient magnitude. Figures 15 and 16 of~\cite{larose2025spline} illustrate this approach. Naturally, this approach cannot be taken with true 3D data, as the coordinates $x$, $y$, and $z$ are required to represent the physical location in the spatial domain, and a fourth coordinate $w$ is required to represent the gradient magnitude. We are unaware of any 4D visualization software which could effectively illustrate the resulting surface, and visualize its discontinuities (or lack thereof). Of course, this is theoretically possible, as one could construct a video where time would be our fourth dimension. Here, we would expect that the gradient magnitude of the WF process would look smoother and more continuous in time between frames, while the gradient magnitude of the $L_2$-projection would be choppy with distinctive jumps between frames. Regardless, even without this more sophisticated approach, we believe our present study still shows the qualitative improvement of visualizing the solution and its derivatives with the WF process compared to $L_2$-projection.   

\pagebreak
\clearpage

\section{Conclusion}
\label{sec;Conclusion}
This work presents, to the best of our knowledge, the most comprehensive and explicit implementation details for the Worsey-Farin spline space to date. We present diagrams and procedures for how to perform the WF-split, and in the Appendix, provide all B-coefficients with equations included and multiple views of each subtetrahedra with its corresponding domain points. 

We also demonstrated the capabilities of the WF space in a solution transfer method. Our method is capable of enforcing conservation (asymptotically); it provides a smooth, continuous surrogate solution; and it can accurately reconstruct the solution and its derivatives for transfer between meshes. For our transfer method, we derived an error estimate which predicts an order of accuracy of $\polyorder+1$ when $1 \leq \polyorder \leq 3$. In addition, numerical experiments were conducted to test the conservation, accuracy, and visualization properties of the WF solution transfer method.  

The order of accuracy experiments showed that the WF transfer method converged at a rate of second order for $\polyorder=1$ and third order for $\polyorder=2$, for the transcendental functions tested. Additionally, the derivatives of the solution converged at one order lower than the solution, i.e.~first order for $\polyorder=1$ and second order for $\polyorder=2$. WF transfer outperformed the standard linear interpolation procedure significantly for both the functions tested. These results reinforce the theoretical error estimate derived in this work, and demonstrate the viability of WF-based approaches for solution transfer applications. As a bonus feature, the visualization study qualitatively showed the improved ability of the smooth WF solution to visualize derivatives compared to the discontinuous visualization provided by $L_2$-projection. 

We also demonstrated an adaptive quadrature routine on tetrahedra, which we used to boost the conservation properties of the WF transfer process on coarser grids. When used in combination with a global $L_2$-projection which solves for all B-coefficients on the grid, we were able to dramatically improve the conservation of the WF transfer method. Based on this work, we recommend exploiting the modular construction of the WF method: in particular, the user can choose $\mathcal{C}^1$-continuity for better visualization properties by computing the spline coefficients using the algebraic formulas of Section~\ref{sec;Implementation}, or choose $\mathcal{C}^{0}$-continuity for better conservation properties by computing the spline coefficients using the global $L_2$-projection procedure of Section~\ref{spline_error}. More generally speaking, the main strength of the proposed WF-based solution transfer method is its flexibility, which makes it potentially competitive with multiple transfer methods across many different settings.

\pagebreak
\clearpage

\section*{CRediT authorship contribution statement}

\textbf{L. Larose}: Writing – original draft, Writing – review and editing, Visualization, Methodology, Formal analysis. \textbf{D.M. Williams}: Writing – original draft, Writing – review and editing, Conceptualization, Formal
analysis, Supervision, Funding acquisition.

\section*{Software Availability}

The solution transfer algorithms were implemented as an extension of the JENRE$^{\text{\textregistered}}$  Multiphysics Framework~\cite{Cor19_SCITECH}, which is United States government-owned software developed by the Naval Research Laboratory with other collaborating institutions. This software is not available for public use or dissemination.

\section*{Declaration of Competing Interests}

The authors declare that they have no known competing financial interests or personal relationships that could have appeared to influence the work reported in this paper.

\section*{Funding}

This research received funding from the United States Naval Research Laboratory (NRL) under grant number N00173-22-2-C008. In turn, the NRL grant itself was funded by Steven Martens, Program Officer for the Power, Propulsion and Thermal Management Program, Code 35, in the United States Office of Naval Research.

\pagebreak

\appendix

\section{Worsey-Farin Implementation Details}
\label{sec;Appendix}
Let us introduce the spline coefficients.
Towards this end, we use $\sigma$ to denote barycentric coordinates on 2D triangular faces and $\kappa$ to denote barycentric coordinates on 3D tetrahedra. In particular, we let 
\begin{align}
\left\{\sigma_i^{j}\right\}_{i=1,\ldots, 3} = \left(\sigma_1^{j}, \sigma_2^{j}, \sigma_3^{j} \right),
\end{align}
denote a vector of barycentric coordinates for the $j$th split point taken with respect to the triangular face on which that split point lies. Similarly, 
\begin{align}
\left\{\kappa_i^{91}\right\}_{i=1,\ldots, 4} = \left(\kappa_1^{91}, \kappa_2^{91}, \kappa_3^{91}, \kappa_4^{91} \right),
\end{align}
denotes a vector of the barycentric coordinates for the incenter point, domain point 91, taken with respect to the macrotetrahedron. In addition, $u_e$, refers to the vector which points from the edge midpoint $\eta_e$ to the opposite vertex, where edge $e$ is formed by vertex points $\vertex^i$ and $\vertex^{i+1}$, and the opposite vertex is $\vertex^{m}$, where $m = i+2$ modulo 3 --- with the exception that anything that returns 0 is mapped to 3.

Here, we present the coefficients positioned at their corresponding domain points for each of the four subtetrahedra in exploded views and `shell' views in Figures~\ref{fig:tet_123_exploded}---\ref{fig:tet_214_shells}. Additionally, we explicitly state the nodes and coefficients which define the WF splines.

\vspace{12pt}

\begin{figure}[h!]
    \centering
    \includegraphics[width=1.0\linewidth]{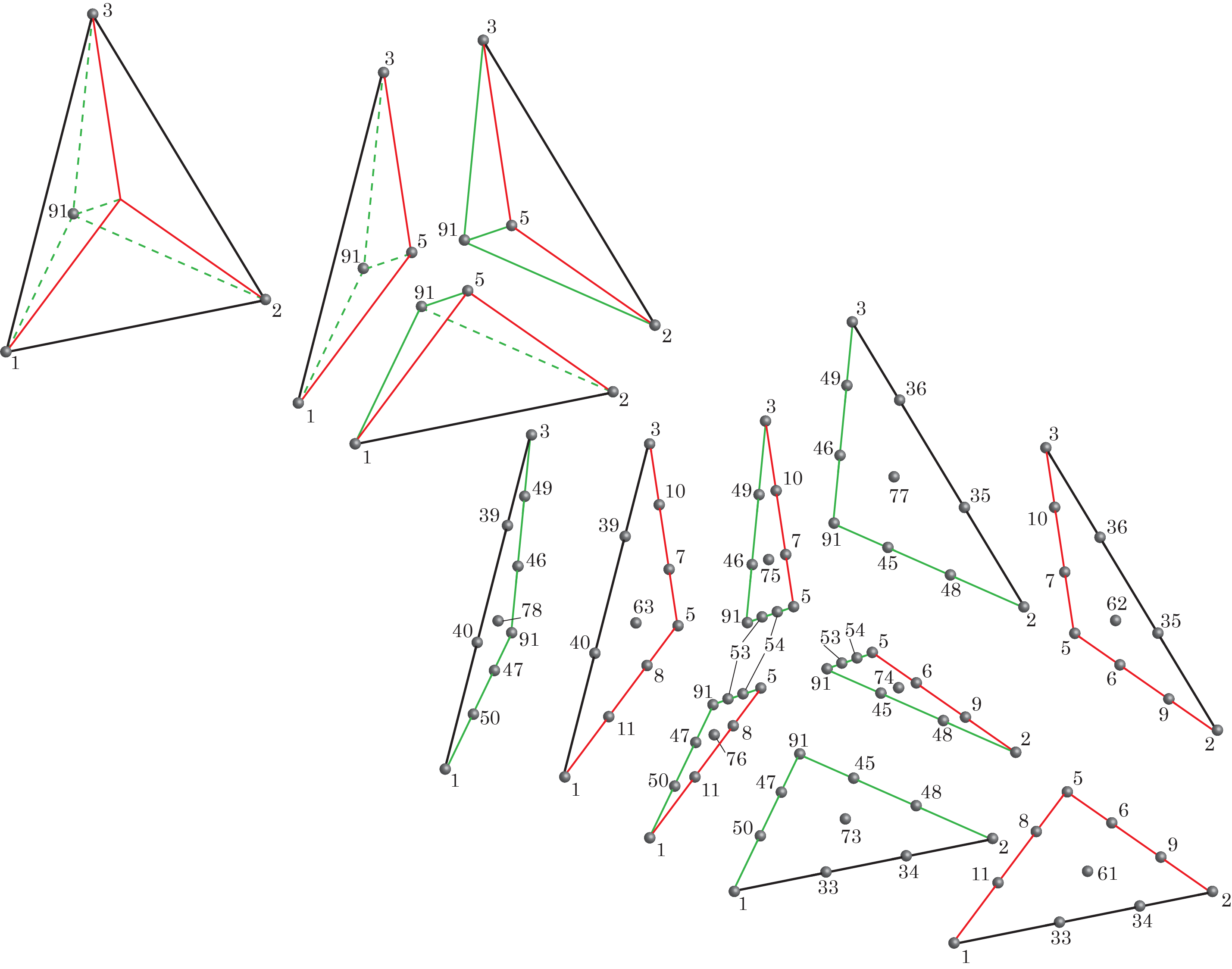}
    \caption{An exploded view of the subtetrahedron and subsequent subtetrahedra formed by vertices $\vertex^1,\vertex^2,\vertex^3$.}
    \label{fig:tet_123_exploded}
\end{figure}
\begin{figure}[h!]
    \centering
    \includegraphics[width=1.0\linewidth]{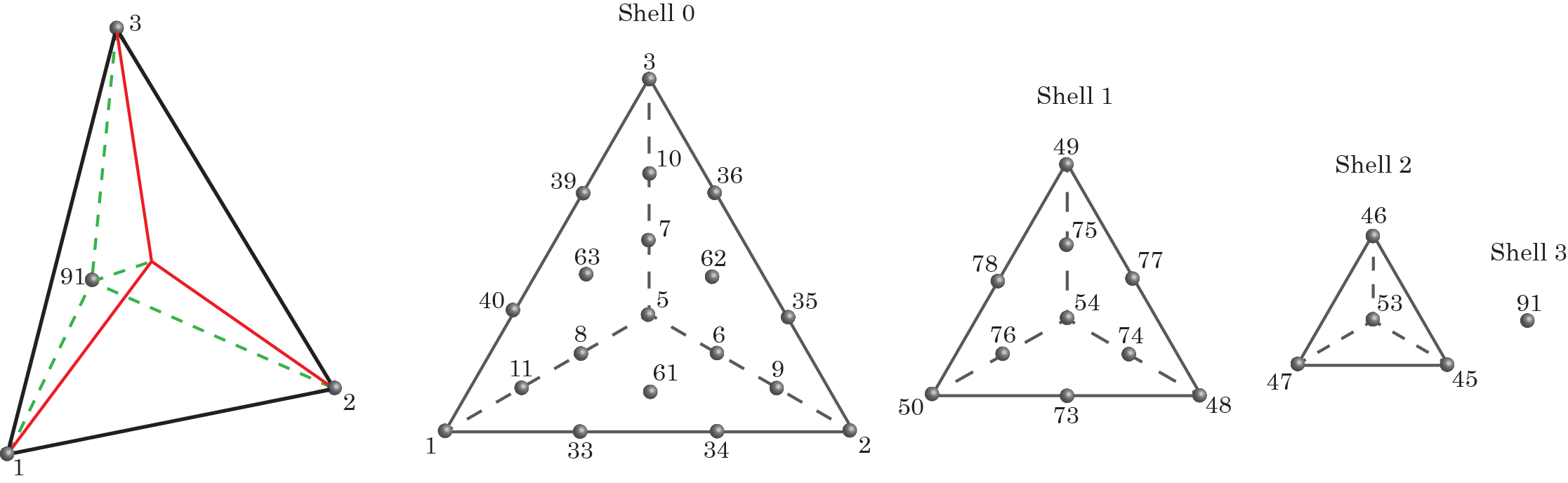}
    \caption{Reorganizing the coefficients of the subtetrahedron and subsequent subtetrahedra formed by vertices $\vertex^1,\vertex^2,\vertex^3$ into a `shell' format.}
    \label{fig:tet_123_shells}
\end{figure}

\noindent Nodes:
\begin{align*}
    &\vertex^5 =\sigma_{1}^{5}\vertex^{1}+\sigma_{2}^{5}\vertex^{2}+\sigma_{3}^{5}\vertex^{3}, \quad \vertex^{91} = \kappa_{1}^{91}\vertex^{1}+\kappa_{2}^{91}\vertex^{2}+\kappa_{3}^{91}\vertex^{4}+\kappa_{4}^{91}\vertex^{3}.
\end{align*}

\noindent Coefficients for the front subtetrahedron (123):
{\allowdisplaybreaks
\begin{align*}
\text{Shell 0:}\\
    &c_{1}  f(\vertex^{1}), \quad c_{2} = f(\vertex^{2}), \quad
    c_{3} = f(\vertex^{3}),\\[1.0ex]
    &c_{5} = \sigma_{1}^{5}c_{8}+\sigma_{2}^{5}c_{6}+\sigma_{3}^{5}c_{7}, \quad
    c_{6} = \sigma_{1}^{5}c_{61}+\sigma_{2}^{5}c_{9}+\sigma_{3}^{5}c_{62},\\[1.0ex]
    &c_{7} = \sigma_{1}^{5}c_{63}+\sigma_{2}^{5}c_{62}+\sigma_{3}^{5}c_{10}, \quad c_{8} = \sigma_{1}^{5}c_{11}+\sigma_{2}^{5}c_{61}+\sigma_{3}^{5}c_{63},\\[1.0ex]
    &c_{9} = \frac{D f(\vertex^{2})\cdot \overrightarrow{\vertex^{2}\vertex^{5}}}{3}+f(\vertex^{2}), \quad
    c_{10} = \frac{D f(\vertex^{3})\cdot \overrightarrow{\vertex^{3}\vertex^{5}}}{3}+f(\vertex^{3}), \quad c_{11} = \frac{D f(\vertex^{1})\cdot \overrightarrow{\vertex^{1}\vertex^{5}}}{3}+f(\vertex^{1}),\\
    &c_{33} = \frac{D f(\vertex^{1})\cdot \overrightarrow{\vertex^{1}\vertex^{2}}}{3}+f(\vertex^{1}), \quad c_{34} = \frac{D f(\vertex^{2})\cdot \overrightarrow{\vertex^{2}\vertex^{1}}}{3}+f(\vertex^{2}), \quad c_{35} = \frac{D f(\vertex^{2})\cdot \overrightarrow{\vertex^{2}\vertex^{3}}}{3}+f(\vertex^{2}),\\[1.0ex]
    &c_{36} = \frac{D f(\vertex^{3})\cdot \overrightarrow{\vertex^{3}\vertex^{2}}}{3}+f(\vertex^{3}), \quad
    c_{39} = \frac{D f(\vertex^{3})\cdot \overrightarrow{\vertex^{3}\vertex^{1}}}{3}+f(\vertex^{3}), \quad c_{40} = \frac{D f(\vertex^{1})\cdot \overrightarrow{\vertex^{1}\vertex^{3}}}{3}+f(\vertex^{1}), \\
    &c_{61}  = \frac{4}{6a_{3}}D_{u_{e}}f(\eta_{e})-\frac{1}{2}(c_{11}+c_{9})-\frac{a_{1}}{2a_{3}}(c_{1}+2c_{33}+c_{34})-\frac{a_{2}}{2a_{3}}(c_{2}+2c_{34}+c_{33}),\\
    &\text{where $a = (-1/2,-1/2,1)$ and $\eta_{e} = (1/2,1/2,0)$},\\[1.0ex]
    &c_{62}  = \frac{4}{6a_{3}}D_{u_{e}}f(\eta_{e})-\frac{1}{2}(c_{9}+c_{10})-\frac{a_{1}}{2a_{3}}(c_{2}+2c_{35}+c_{36})-\frac{a_{2}}{2a_{3}}(c_{3}+2c_{36}+c_{35}),\\
    &\text{where $a = (-1/2,-1/2,1)$ and $\eta_{e} = (1/2,1/2,0)$},\\[1.0ex]
    &c_{63}  = \frac{4}{6a_{3}}D_{u_{e}}f(\eta_{e})-\frac{1}{2}(c_{10}+c_{11})-\frac{a_{1}}{2a_{3}}(c_{3}+2c_{39}+c_{40})-\frac{a_{2}}{2a_{3}}(c_{1}+2c_{40}+c_{39}),\\
    &\text{where $a = (-1/2,-1/2,1)$ and $\eta_{e} = (1/2,1/2,0)$}.\\[1.0ex]
    \text{Shell 1:}\\
    &c_{54}  = \sigma_{1}^{5} c_{76}+\sigma_{2}^{5} c_{74}+\sigma_{3}^{5} c_{75},\quad c_{74} = \sigma_{1}^5 c_{73}+\sigma_{2}^5 c_{48}+\sigma_{3}^5 c_{77},\\
    &c_{75}  = \sigma_{1}^5 c_{78}+\sigma_{2}^5 c_{77}+\sigma_{3}^5 c_{49}, \quad c_{76} = \sigma_{1}^5 c_{50}+\sigma_{2}^5 c_{73}+\sigma_{3}^5 c_{78},\\
    &c_{48} = \frac{D f(\vertex^{2})\cdot \overrightarrow{\vertex^{2}\vertex^{91}}}{3}+f(\vertex^{2}), \quad
    c_{49} = \frac{D f(\vertex^{3})\cdot \overrightarrow{\vertex^{3}\vertex^{91}}}{3}+f(\vertex^{3}), \quad c_{50} = \frac{D f(\vertex^{1})\cdot \overrightarrow{\vertex^{1}\vertex^{91}}}{3}+f(\vertex^{1}),\\[1.0ex]
    &c_{73}  = \frac{4}{6a_{3}}D_{u_{e}}f(\eta_{e})-\frac{1}{2}(c_{50}+c_{48})-\frac{a_{1}}{2a_{3}}(c_{1}+2c_{33}+c_{34})-\frac{a_{2}}{2a_{3}}(c_{2}+2c_{34}+c_{33}),\\
    &\text{where $a = (-1/2,-1/2,1)$ and $\eta_{e} = (1/2,1/2,0)$},\\
    &c_{77}  = \frac{4}{6a_{3}}D_{u_{e}}f(\eta_{e})-\frac{1}{2}(c_{48}+c_{49})-\frac{a_{1}}{2a_{3}}(c_{2}+2c_{35}+c_{36})-\frac{a_{2}}{2a_{3}}(c_{3}+2c_{36}+c_{35}),\\
    &\text{where $a = (-1/2,-1/2,1)$ and $\eta_{e} = (1/2,1/2,0)$},\\
    &c_{78}  = \frac{4}{6a_{3}}D_{u_{e}}f(\eta_{e})-\frac{1}{2}(c_{49}+c_{50})-\frac{a_{1}}{2a_{3}}(c_{3}+2c_{39}+c_{40})-\frac{a_{2}}{2a_{3}}(c_{1}+2c_{40}+c_{39}),\\
    &\text{where $a = (-1/2,-1/2,1)$ and $\eta_{e} = (1/2,1/2,0)$}.\\
    \text{Shell 2:}\\
    &c_{53} = \sigma_{1}^{5} c_{47}+\sigma_{2}^{5} c_{45}+\sigma_{3}^{5} c_{46},\\
    &c_{45}  = \kappa_{1}^{91} c_{73}+\kappa_{2}^{91} c_{48}+\kappa_{3}^{91} c_{80}+\kappa_{4}^{91} c_{77}, \quad c_{46} = \kappa_{1}^{91} c_{78}+\kappa_{2}^{91} c_{77}+\kappa_{3}^{91} c_{79}+\kappa_{4}^{91} c_{49},\\
    &c_{47}  = \kappa_{1}^{91} c_{50}+\kappa_{2}^{91} c_{73}+\kappa_{3}^{91} c_{87}+\kappa_{4}^{91} c_{78}.\\
    \text{Shell 3:}\\
    &c_{91}  = \kappa_{1}^{91} c_{47}+\kappa_{2}^{91} c_{45}+\kappa_{3}^{91} c_{51}+\kappa_{4}^{91} c_{46}.\\
\end{align*}
}
\begin{figure}[h!]
    \centering
    \includegraphics[width=1.0\linewidth]{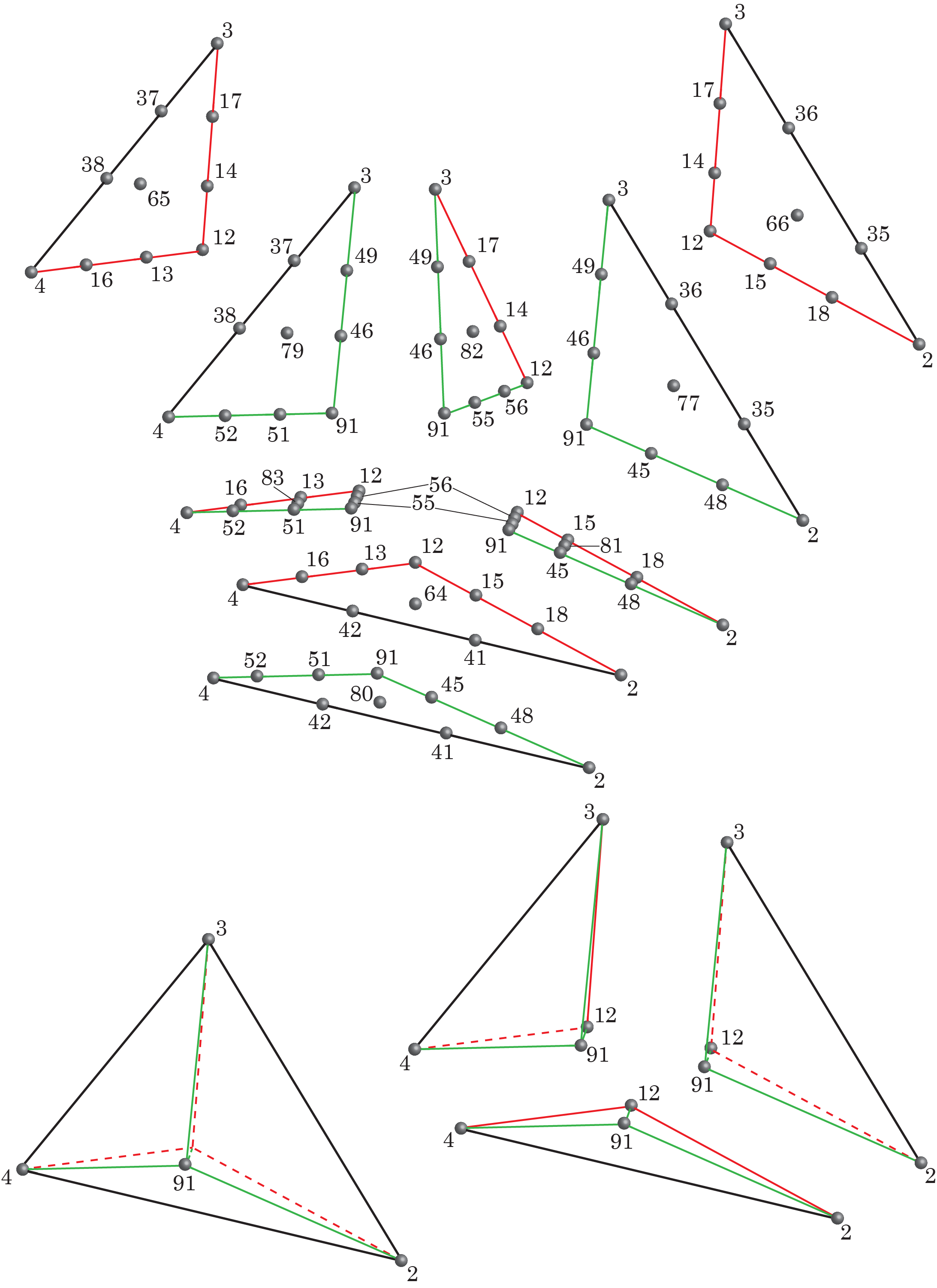}
    \caption{An exploded view of the subtetrahedron and subsequent subtetrahedra formed by vertices $\vertex^2,\vertex^4,\vertex^3$.}
    \label{fig:tet_243_exploded}
\end{figure}
\begin{figure}[h]
    \centering
    \includegraphics[width=1.0\linewidth]{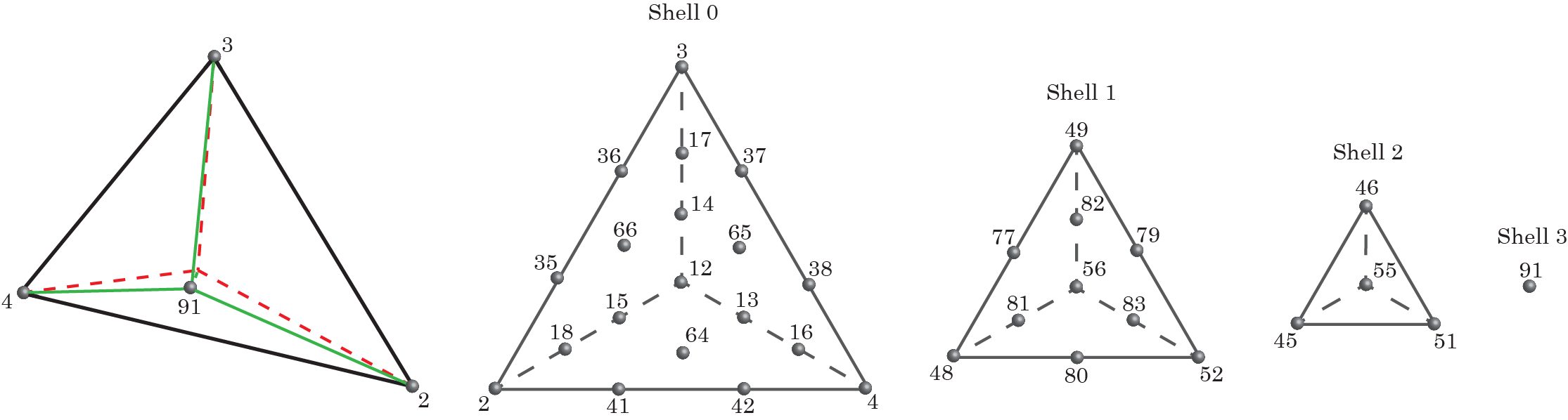}
    \caption{Reorganizing the coefficients of the subtetrahedron and subsequent subtetrahedra formed by vertices $\vertex^2,\vertex^4,\vertex^3$ into a `shell' format.}
    \label{fig:tet_243_shells}
\end{figure}

\noindent Coefficients for the back subtetrahedron (243):\\
\noindent Nodes:
\begin{align*}
    &\vertex^{12} =\sigma_{1}^{12}\vertex^{2}+\sigma_{2}^{12}\vertex^{4}+\sigma_{3}^{12}\vertex^{3}, \quad \vertex^{91} = \kappa_{1}^{91}\vertex^{1}+\kappa_{2}^{91}\vertex^{2}+\kappa_{3}^{91}\vertex^{4}+\kappa_{4}^{91}\vertex^{3}.
\end{align*}
{\allowdisplaybreaks
\begin{align*}
\text{Shell 0:}\\
    &c_{2} = f(\vertex^{2}), \quad c_{3} = f(\vertex^{3}), \quad c_{4} = f(\vertex^{4}),\\
    &c_{12} = \sigma_{1}^{12}c_{15}+\sigma_{2}^{12}c_{13}+\sigma_{3}^{12}c_{14}, \quad c_{13} = \sigma_{1}^{12}c_{64}+\sigma_{2}^{12}c_{16}+\sigma_{3}^{12}c_{65},\\
    &c_{14} = \sigma_{1}^{12}c_{66}+\sigma_{2}^{12}c_{65}+\sigma_{3}^{12}c_{17}, \quad c_{15} = \sigma_{1}^{12}c_{18}+\sigma_{2}^{12}c_{64}+\sigma_{3}^{12}c_{66},\\
    &c_{16} = \frac{D f(\vertex^{4})\cdot \overrightarrow{\vertex^{4}\vertex^{12}}}{3}+f(\vertex^{4}), \quad c_{17} = \frac{D f(\vertex^{3})\cdot \overrightarrow{\vertex^{3}\vertex^{12}}}{3}+f(\vertex^{3}), \quad c_{18} = \frac{D f(\vertex^{2})\cdot \overrightarrow{\vertex^{1}\vertex^{12}}}{3}+f(\vertex^{2}),\\
    &c_{35} = \frac{D f(\vertex^{2})\cdot \overrightarrow{\vertex^{2}\vertex^{3}}}{3}+f(\vertex^{2}), \quad c_{36} = \frac{D f(\vertex^{3})\cdot \overrightarrow{\vertex^{3}\vertex^{2}}}{3}+f(\vertex^{3}), \quad c_{37} = \frac{D f(\vertex^{3})\cdot \overrightarrow{\vertex^{3}\vertex^{4}}}{3}+f(\vertex^{3}),\\
    &c_{38} = \frac{D f(\vertex^{4})\cdot \overrightarrow{\vertex^{4}\vertex^{3}}}{3}+f(\vertex^{4}), \quad c_{41} = \frac{D f(\vertex^{2})\cdot \overrightarrow{\vertex^{2}\vertex^{4}}}{3}+f(\vertex^{2}), \quad c_{42} = \frac{D f(\vertex^{4})\cdot \overrightarrow{\vertex^{4}\vertex^{2}}}{3}+f(\vertex^{4}),\\ 
    &c_{64}  = \frac{4}{6a_{3}}D_{u_{e}}f(\eta_{e})-\frac{1}{2}(c_{18}+c_{16})-\frac{a_{1}}{2a_{3}}(c_{2}+2c_{41}+c_{42})-\frac{a_{2}}{2a_{3}}(c_{4}+2c_{42}+c_{41}),\\
    &\text{where $a = (-1/2,-1/2,1)$ and $\eta_{e} = (1/2,1/2,0)$},\\
    &c_{65}  = \frac{4}{6a_{3}}D_{u_{e}}f(\eta_{e})-\frac{1}{2}(c_{16}+c_{17})-\frac{a_{1}}{2a_{3}}(c_{4}+2c_{38}+c_{37})-\frac{a_{2}}{2a_{3}}(c_{3}+2c_{37}+c_{38}),\\
    &\text{where $a = (-1/2,-1/2,1)$ and $\eta_{e} = (1/2,1/2,0)$},\\
    &c_{66}  = \frac{4}{6a_{3}}D_{u_{e}}f(\eta_{e})-\frac{1}{2}(c_{17}+c_{18})-\frac{a_{1}}{2a_{3}}(c_{3}+2c_{36}+c_{35})-\frac{a_{2}}{2a_{3}}(c_{2}+2c_{35}+c_{36}),\\
    &\text{where $a = (-1/2,-1/2,1)$ and $\eta_{e} = (1/2,1/2,0)$}.\\
    \text{Shell 1:}\\
    &c_{56}  = \sigma_{1}^{12} c_{81}+\sigma_{2}^{12} c_{83}+\sigma_{3}^{12} c_{82}, \quad c_{81} = \sigma_{1}^{12} c_{48}+\sigma_{2}^{12} c_{80}+\sigma_{3}^{12} c_{77},\\
    &c_{82}  = \sigma_{1}^{12} c_{77}+\sigma_{2}^{12} c_{79}+\sigma_{3}^{12} c_{49}, \quad c_{83} = \sigma_{1}^{12} c_{80}+\sigma_{2}^{12} c_{52}+\sigma_{3}^{12} c_{79},\\
    &c_{48} = \frac{D f(\vertex^{2})\cdot \overrightarrow{\vertex^{2}\vertex^{91}}}{3}+f(\vertex^{2}), \quad c_{49} = \frac{D f(\vertex^{3})\cdot \overrightarrow{\vertex^{3}\vertex^{91}}}{3}+f(\vertex^{3}), \quad c_{52} = \frac{D f(\vertex^{4}),\cdot \overrightarrow{\vertex^{4}\vertex^{91}}}{3}+f(\vertex^{4}),\\
    &c_{77}  = \frac{4}{6a_{3}}D_{u_{e}}f(\eta_{e})-\frac{1}{2}(c_{49}+c_{48})-\frac{a_{1}}{2a_{3}}(c_{3}+2c_{36}+c_{35})-\frac{a_{2}}{2a_{3}}(c_{2}+2c_{35}+c_{36}),\\
    &\text{where $a = (-1/2,-1/2,1)$ and $\eta_{e} = (1/2,1/2,0)$},\\
    &c_{79}  = \frac{4}{6a_{3}}D_{u_{e}}f(\eta_{e})-\frac{1}{2}(c_{52}+c_{49})-\frac{a_{1}}{2a_{3}}(c_{4}+2c_{38}+c_{37})-\frac{a_{2}}{2a_{3}}(c_{3}+2c_{37}+c_{38}),\\
    &\text{where $a = (1,-1/2,-1/2)$ and $\eta_{e} = (0,1/2,1/2)$},\\
    &c_{80}  = \frac{4}{6a_{3}}D_{u_{e}}f(\eta_{e})-\frac{1}{2}(c_{48}+c_{52})-\frac{a_{1}}{2a_{3}}(c_{2}+2c_{41}+c_{42})-\frac{a_{2}}{2a_{3}}(c_{4}+2c_{42}+c_{41}),\\
    &\text{where $a = (-1/2,-1/2,1)$ and $\eta_{e} = (1/2,1/2,0)$}.\\
    \text{Shell 2:}\\
    &c_{55} = \sigma_{1}^{12} c_{45}+\sigma_{2}^{12} c_{51}+\sigma_{3}^{12} c_{46},\\
    &c_{45}  = \kappa_{1}^{91} c_{73}+\kappa_{2}^{91} c_{48}+\kappa_{3}^{91} c_{80}+\kappa_{4}^{91} c_{77}, \quad c_{46} = \kappa_{1}^{91} c_{78}+\kappa_{2}^{91} c_{77}+\kappa_{3}^{91} c_{79}+\kappa_{4}^{91} c_{49},\\
    &c_{51}  = \kappa_{1}^{91} c_{87}+\kappa_{2}^{91} c_{80}+\kappa_{3}^{91} c_{52}+\kappa_{4}^{91} c_{79}.\\
    \text{Shell 3:}\\
    &c_{91}  = \kappa_{1}^{91} c_{47}+\kappa_{2}^{91} c_{45}+\kappa_{3}^{91} c_{51}+\kappa_{4}^{91} c_{46}.\\
\end{align*}
}
\begin{figure}[h!]
    \centering
    \includegraphics[width=1.0\linewidth]{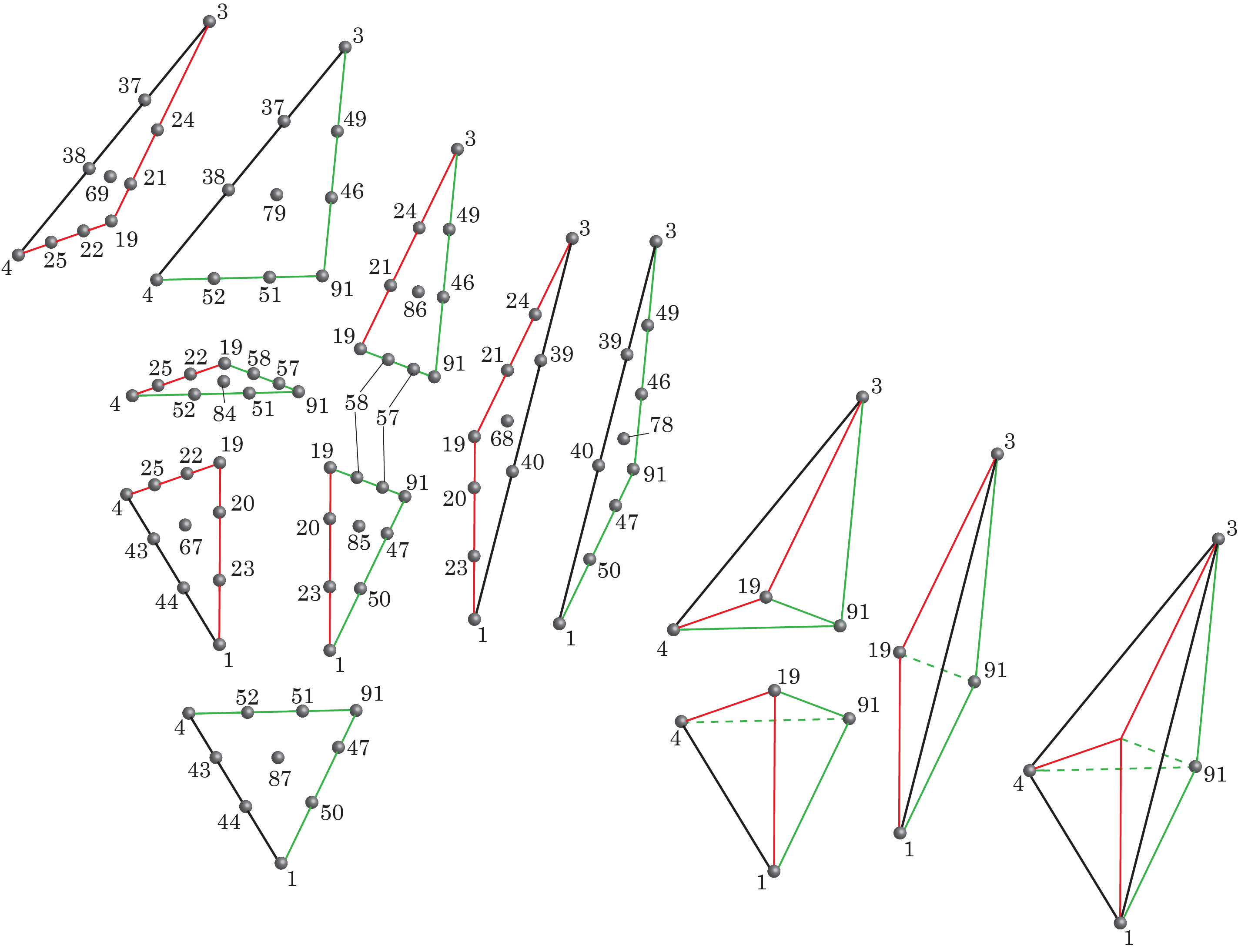}
    \caption{An exploded view of the subtetrahedron and subsequent subtetrahedra formed by vertices $\vertex^4,\vertex^1,\vertex^3$.}
    \label{fig:tet_413_exploded}
\end{figure}
\begin{figure}[h]
    \centering
    \includegraphics[width=1.0\linewidth]{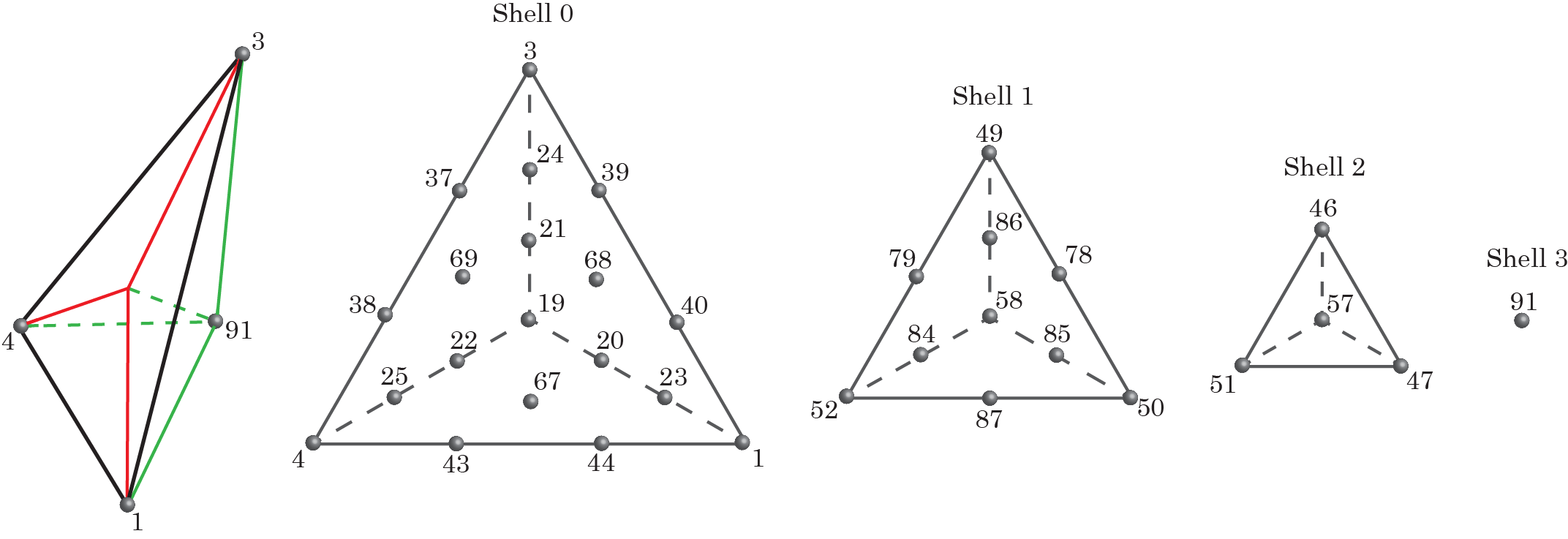}
    \caption{Reorganizing the coefficients of the subtetrahedron and subsequent subtetrahedra formed by vertices $\vertex^4,\vertex^1,\vertex^3$ into a `shell' format.}
    \label{fig:tet_413_shells}
\end{figure}

\noindent Coefficients for the side subtetrahedron (413):\\
\noindent Nodes:
\begin{align*}
    &\vertex^{19} =\sigma_{1}^{19}\vertex^{4}+\sigma_{2}^{19}\vertex^{1}+\sigma_{3}^{19}\vertex^{3}, \quad
    \vertex^{91} = \kappa_{1}^{91}\vertex^{1}+\kappa_{2}^{91}\vertex^{2}+\kappa_{3}^{91}\vertex^{4}+\kappa_{4}^{91}\vertex^{3}.
\end{align*}
{\allowdisplaybreaks
\begin{align*}
\text{Shell 0:}\\
    &c_{1} = f(\vertex^{1}), \quad c_{3} = f(\vertex^{2}), \quad c_{4} = f(\vertex^{3}),\\
    &c_{19} = \sigma_{1}^{19}c_{22}+\sigma_{2}^{19}c_{20}+\sigma_{3}^{19}c_{21}, \quad c_{20} = \sigma_{1}^{19}c_{67}+\sigma_{2}^{19}c_{23}+\sigma_{3}^{19}c_{68},\\
    &c_{21} = \sigma_{1}^{19}c_{69}+\sigma_{2}^{19}c_{68}+\sigma_{3}^{19}c_{24}, \quad c_{22} = \sigma_{1}^{19}c_{25}+\sigma_{2}^{19}c_{67}+\sigma_{3}^{19}c_{69},\\
    &c_{23} = \frac{D f(\vertex^{1})\cdot \overrightarrow{\vertex^{1}\vertex^{19}}}{3}+f(\vertex^{1}), \quad c_{24} = \frac{D f(\vertex^{3})\cdot \overrightarrow{\vertex^{3}\vertex^{19}}}{3}+f(\vertex^{3}), \quad c_{25} = \frac{D f(\vertex^{4})\cdot \overrightarrow{\vertex^{4}\vertex^{19}}}{3}+f(\vertex^{4}),\\
    &c_{37} = \frac{D f(\vertex^{3})\cdot \overrightarrow{\vertex^{3}\vertex^{4}}}{3}+f(\vertex^{3}), \quad c_{38} = \frac{D f(\vertex^{4})\cdot \overrightarrow{\vertex^{4}\vertex^{3}}}{3}+f(\vertex^{4}), \quad c_{39} = \frac{D f(\vertex^{3})\cdot \overrightarrow{\vertex^{3}\vertex^{1}}}{3}+f(\vertex^{3}),\\
    &c_{40} = \frac{D f(\vertex^{1})\cdot \overrightarrow{\vertex^{1}\vertex^{3}}}{3}+f(\vertex^{1}), \quad c_{43} = \frac{D f(\vertex^{4})\cdot \overrightarrow{\vertex^{4}\vertex^{1}}}{3}+f(\vertex^{4}), \quad c_{44} = \frac{D f(\vertex^{1})\cdot \overrightarrow{\vertex^{1}\vertex^{4}}}{3}+f(\vertex^{1}),\\
    &c_{67}  = \frac{4}{6a_{3}}D_{u_{e}}f(\eta_{e})-\frac{1}{2}(c_{25}+c_{23})-\frac{a_{1}}{2a_{3}}(c_{4}+2c_{43}+c_{44})-\frac{a_{2}}{2a_{3}}(c_{1}+2c_{44}+c_{43}),\\
    &\text{where $a = (-1/2,-1/2,1)$ and $\eta_{e} = (1/2,1/2,0)$},\\
    &c_{68}  = \frac{4}{6a_{3}}D_{u_{e}}f(\eta_{e})-\frac{1}{2}(c_{23}+c_{24})-\frac{a_{1}}{2a_{3}}(c_{1}+2c_{40}+c_{39})-\frac{a_{2}}{2a_{3}}(c_{3}+2c_{39}+c_{40}),\\
    &\text{where $a = (-1/2,-1/2,1)$ and $\eta_{e} = (1/2,1/2,0)$},\\
    &c_{69} = \frac{4}{6a_{3}}D_{u_{e}}f(\eta_{e})-\frac{1}{2}(c_{24}+c_{25})-\frac{a_{1}}{2a_{3}}(c_{3}+2c_{37}+c_{38})-\frac{a_{2}}{2a_{3}}(c_{4}+2c_{38}+c_{37}),\\
    &\text{where $a = (-1/2,-1/2,1)$ and $\eta_{e} = (1/2,1/2,0)$}.\\
    \text{Shell 1:}\\
    &c_{58} = \sigma_{1}^{19} c_{84}+\sigma_{2}^{19} c_{85}+\sigma_{3}^{19} c_{86}, \quad c_{84} = \sigma_{1}^{19}c_{52}+\sigma_{2}^{19}c_{87}+\sigma_{3}^{19}c_{79},\\
    &c_{85} = \sigma_{1}^{19}c_{87}+\sigma_{2}^{19}c_{50}+\sigma_{3}^{19}c_{78}, \quad c_{86}  = \sigma_{1}^{19}c_{79}+\sigma_{2}^{19}c_{78}+\sigma_{3}^{19}c_{49},\\
    &c_{49} = \frac{D f(\vertex^{3})\cdot \overrightarrow{\vertex^{3}\vertex^{91}}}{3}+f(\vertex^{3}), \quad c_{50} = \frac{D f(\vertex^{1})\cdot \overrightarrow{\vertex^{1}\vertex^{91}}}{3}+f(\vertex^{1}), \quad c_{52} = \frac{D f(\vertex^{4})\cdot \overrightarrow{\vertex^{4}\vertex^{91}}}{3}+f(\vertex^{4}),\\
    &c_{78} = \frac{4}{6a_{3}}D_{u_{e}}f(\eta_{e})-\frac{1}{2}(c_{85}+c_{86})-\frac{a_{1}}{2a_{3}}(c_{1}+2c_{40}+c_{39})-\frac{a_{2}}{2a_{3}}(c_{3}+2c_{39}+c_{40}),\\
    &\text{where $a = (-1/2,-1/2,1)$ and $\eta_{e} = (1/2,1/2,0)$},\\
    &c_{79} = \frac{4}{6a_{3}}D_{u_{e}}f(\eta_{e})-\frac{1}{2}(c_{49}+c_{52})-\frac{a_{1}}{2a_{3}}(c_{3}+2c_{37}+c_{38})-\frac{a_{2}}{2a_{3}}(c_{4}+2c_{38}+c_{37}),\\
    &\text{where $a = (-1/2,-1/2,1)$ and $\eta_{e} = (1/2,1/2,0)$},\\
    &c_{87} = \frac{4}{6a_{3}}D_{u_{e}}f(\eta_{e})-\frac{1}{2}(c_{25}+c_{23})-\frac{a_{1}}{2a_{3}}(c_{4}+2c_{43}+c_{44})-\frac{a_{2}}{2a_{3}}(c_{1}+2c_{44}+c_{43}),\\
    &\text{where $a = (-1/2,-1/2,1)$ and $\eta_{e} = (1/2,1/2,0)$}.\\
    \text{Shell 2:}\\
    &c_{57} = \sigma_{1}^{19} c_{51}+\sigma_{2}^{19} c_{47}+\sigma_{3}^{19} c_{46},\\
    &c_{46} = \kappa_{1}^{91} c_{78}+\kappa_{2}^{91} c_{77}+\kappa_{3}^{91} c_{79}+\kappa_{4}^{91} c_{49}, \quad c_{47} = \kappa_{1}^{91} c_{50}+\kappa_{2}^{91} c_{73}+\kappa_{3}^{91} c_{87}+\kappa_{4}^{91} c_{78},\\
    &c_{51} = \kappa_{1}^{91} c_{87}+\kappa_{2}^{91} c_{80}+\kappa_{3}^{91} c_{52}+\kappa_{4}^{91} c_{79}.\\
    \text{Shell 3:}\\
    &c_{91} = \kappa_{1}^{91} c_{47}+\kappa_{2}^{91} c_{45}+\kappa_{3}^{91} c_{51}+\kappa_{4}^{91} c_{46}.\\
\end{align*}
}
\begin{figure}[h!]
    \centering
    \includegraphics[width=1.0\linewidth]{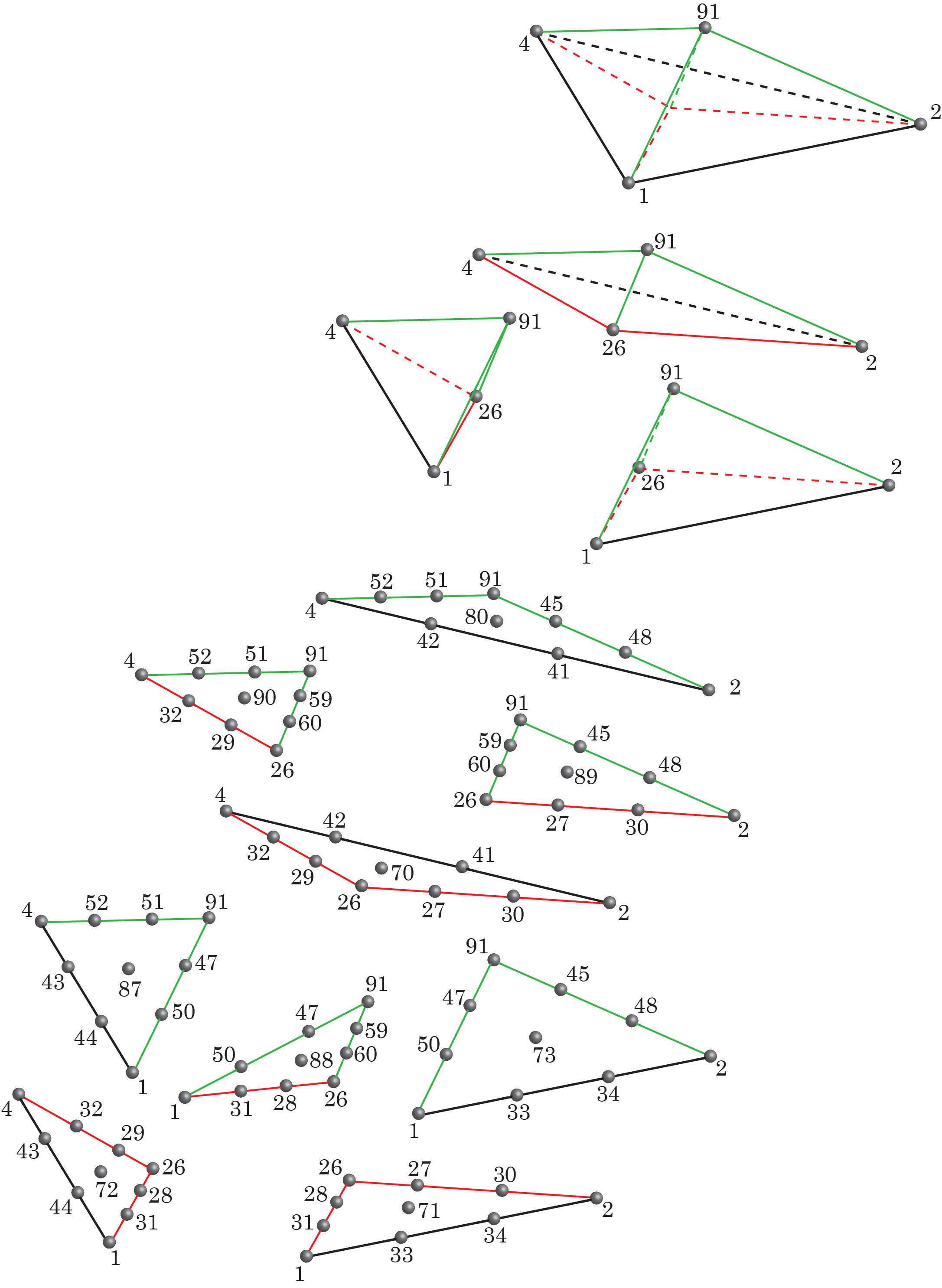}
    \caption{An exploded view of the subtetrahedron and subsequent subtetrahedra formed by vertices $\vertex^2,\vertex^1,\vertex^4$.}
    \label{fig:tet_214_exploded}
\end{figure}
\begin{figure}[h]
    \centering
    \includegraphics[width=1.0\linewidth]{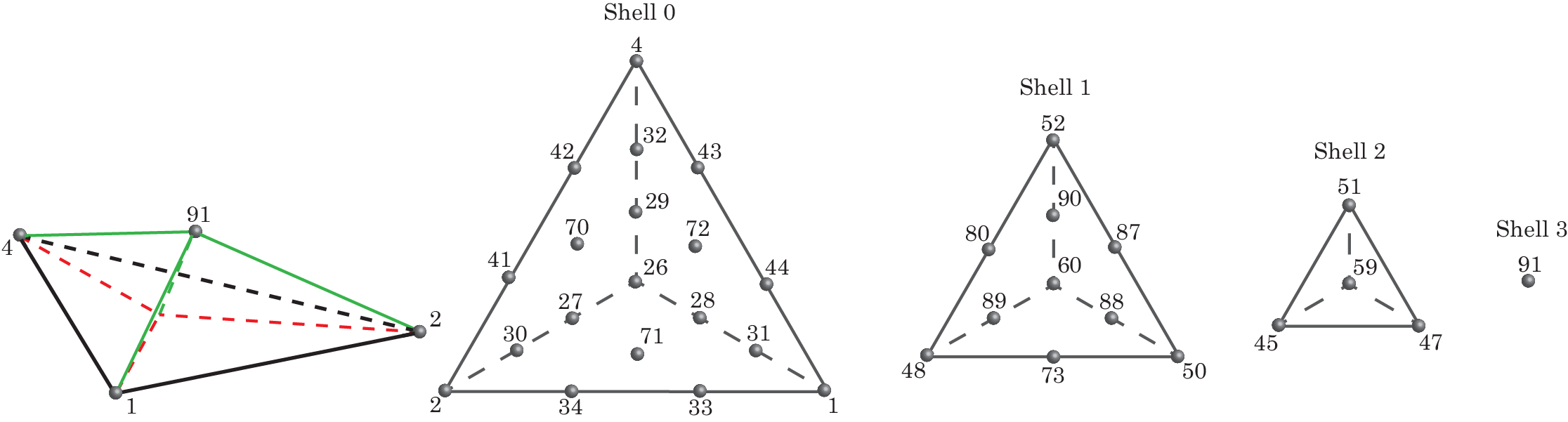}
    \caption{Reorganizing the coefficients of the subtetrahedron and subsequent subtetrahedra formed by vertices $\vertex^2,\vertex^1,\vertex^4$ into a `shell' format.}
    \label{fig:tet_214_shells}
\end{figure}

\noindent Bottom subtetrahedron (214):\\
\noindent Nodes:
\begin{align*}
    &\vertex^{26} =\sigma_{1}^{26}\vertex^{2}+\sigma_{2}^{26}\vertex^{1}+\sigma_{3}^{26}\vertex^{4}, \quad \vertex^{91} = \kappa_{1}^{91}\vertex^{1}+\kappa_{2}^{91}\vertex^{2}+\kappa_{3}^{91}\vertex^{4}+\kappa_{4}^{91}\vertex^{3}.
\end{align*}
{\allowdisplaybreaks
\begin{align*}
\text{Shell 0:}\\
    &c_{1} = f(\vertex^{1}), \quad c_{2} = f(\vertex^{2}), \quad c_{4} = f(\vertex^{4}),\\
    &c_{26} = \sigma_{1}^{26}c_{27}+\sigma_{2}^{26}c_{28}+\sigma_{3}^{26}c_{29}, \quad c_{27} = \sigma_{1}^{26}c_{30}+\sigma_{2}^{26}c_{71}+\sigma_{3}^{26}c_{70},\\
    &c_{28} = \sigma_{1}^{26}c_{71}+\sigma_{2}^{26}c_{31}+\sigma_{3}^{26}c_{72}, \quad c_{29} = \sigma_{1}^{26}c_{70}+\sigma_{2}^{26}c_{72}+\sigma_{3}^{26}c_{32},\\
    &c_{30} = \frac{D f(\vertex^{2})\cdot \overrightarrow{\vertex^{2}\vertex^{26}}}{3}+f(\vertex^{2}), \quad c_{31} = \frac{D f(\vertex^{1})\cdot \overrightarrow{\vertex^{1}\vertex^{26}}}{3}+f(\vertex^{1}), \quad c_{32} = \frac{D f(\vertex^{4})\cdot \overrightarrow{\vertex^{4}\vertex^{26}}}{3}+f(\vertex^{4}),\\
    &c_{33} = \frac{D f(\vertex^{1})\cdot \overrightarrow{\vertex^{1}\vertex^{2}}}{3}+f(\vertex^{1}), \quad c_{34} = \frac{D f(\vertex^{2})\cdot \overrightarrow{\vertex^{2}\vertex^{1}}}{3}+f(\vertex^{2}), \quad c_{41} = \frac{D f(\vertex^{2})\cdot \overrightarrow{\vertex^{2}\vertex^{4}}}{3}+f(\vertex^{2}),\\
    &c_{42} = \frac{D f(\vertex^{4})\cdot \overrightarrow{\vertex^{4}\vertex^{2}}}{3}+f(\vertex^{4}), \quad c_{43} = \frac{D f(\vertex^{4})\cdot \overrightarrow{\vertex^{4}\vertex^{1}}}{3}+f(\vertex^{4}), \quad c_{44} = \frac{D f(\vertex^{1})\cdot \overrightarrow{\vertex^{1}\vertex^{4}}}{4}+f(\vertex^{1}),\\
    &c_{70} = \frac{4}{6a_{3}}D_{u_{e}}f(\eta_{e})-\frac{1}{2}(c_{32}+c_{30})-\frac{a_{1}}{2a_{3}}(c_{4}+2c_{42}+c_{41})-\frac{a_{2}}{2a_{3}}(c_{2}+2c_{41}+c_{42}),\\
    &\text{where $a = (-1/2,-1/2,1)$ and $\eta_{e} = (1/2,1/2,0)$},\\
    &c_{71} = \frac{4}{6a_{3}}D_{u_{e}}f(\eta_{e})-\frac{1}{2}(c_{30}+c_{31})-\frac{a_{1}}{2a_{3}}(c_{2}+2c_{34}+c_{33})-\frac{a_{2}}{2a_{3}}(c_{1}+2c_{33}+c_{34}),\\
    &\text{where $a = (-1/2,-1/2,1)$ and $\eta_{e} = (1/2,1/2,0)$},\\
    &c_{72} = \frac{4}{6a_{3}}D_{u_{e}}f(\eta_{e})-\frac{1}{2}(c_{31}+c_{32})-\frac{a_{1}}{2a_{3}}(c_{1}+2c_{44}+c_{43})-\frac{a_{2}}{2a_{3}}(c_{4}+2c_{43}+c_{44}),\\
    &\text{where $a = (-1/2,-1/2,1)$ and $\eta_{e} = (1/2,1/2,0)$}.\\
    \text{Shell 1:}\\
    &c_{60} = \sigma_{1}^{26} c_{89}+\sigma_{2}^{26} c_{88}+\sigma_{3}^{26} c_{90}, \quad c_{88} = \sigma_{1}^{26} c_{73}+\sigma_{2}^{26} c_{50}+\sigma_{3}^{26} c_{87},\\
    &c_{89} = \sigma_{1}^{26} c_{48}+\sigma_{2}^{26} c_{73}+\sigma_{3}^{26} c_{80}, \quad c_{90} = \sigma_{1}^{26} c_{80}+\sigma_{2}^{26} c_{87}+\sigma_{3}^{26} c_{52},\\
    &c_{48} = \frac{D f(\vertex^{2})\cdot \overrightarrow{\vertex^{2}\vertex^{91}}}{3}+f(\vertex^{2}), \quad c_{50} = \frac{D f(\vertex^{1})\cdot \overrightarrow{\vertex^{1}\vertex^{91}}}{3}+f(\vertex^{1}), \quad c_{52} = \frac{D f(\vertex^{4})\cdot \overrightarrow{\vertex^{4}\vertex^{91}}}{4}+f(\vertex^{4}),\\
    &c_{73} = \frac{4}{6a_{3}}D_{u_{e}}f(\eta_{e})-\frac{1}{2}(c_{50}+c_{48})-\frac{a_{1}}{2a_{3}}(c_{2}+2c_{34}+c_{33})-\frac{a_{2}}{2a_{3}}(c_{1}+2c_{33}+c_{34}),\\
    &\text{where $a = (-1/2,-1/2,1)$ and $\eta_{e} = (1/2,1/2,0)$},\\
    &c_{80} = \frac{4}{6a_{2}}D_{u_{e}}f(\eta_{e})-\frac{1}{2}(c_{52}+c_{48})-\frac{a_{1}}{2a_{3}}(c_{4}+2c_{42}+c_{41})-\frac{a_{2}}{2a_{3}}(c_{2}+2c_{41}+c_{42}),\\
    &\text{where $a = (-1/2,-1/2,1)$ and $\eta_{e} = (1/2,1/2,0)$},\\
    &c_{87} = \frac{4}{6a_{1}}D_{u_{e}}f(\eta_{e})-\frac{1}{2}(c_{50}+c_{52})-\frac{a_{1}}{2a_{3}}(c_{1}+2c_{44}+c_{43})-\frac{a_{2}}{2a_{3}}(c_{4}+2c_{43}+c_{44}),\\
    &\text{where $a = (-1/2,-1/2,1)$ and $\eta_{e} = (1/2,1/2,0)$}.\\
    \text{Shell 2:}\\
    &c_{59} = \sigma_{1}^{26} c_{45}+\sigma_{2}^{26} c_{47}+\sigma_{3}^{26} c_{51},\\
    &c_{45} = \kappa_{1}^{91} c_{73}+\kappa_{2}^{91} c_{48}+\kappa_{3}^{91} c_{80}+\kappa_{4}^{91} c_{77}, \quad c_{47} = \kappa_{1}^{91} c_{50}+\kappa_{2}^{91} c_{73}+\kappa_{3}^{91} c_{87}+\kappa_{4}^{91} c_{78},\\
    &c_{51} = \kappa_{1}^{91} c_{87}+\kappa_{2}^{91} c_{80}+\kappa_{3}^{91} c_{52}+\kappa_{4}^{91} c_{79}.\\
    \text{Shell 3:}\\
    &c_{91} = \kappa_{1}^{91} c_{47}+\kappa_{2}^{91} c_{45}+\kappa_{3}^{91} c_{51}+\kappa_{4}^{91} c_{46}.\\
\end{align*}
}
\pagebreak
\clearpage


\end{document}